\documentclass[11pt,reqno]{amsart}
\usepackage{preprintStyle}

\usepackage{adjustbox}
\usepackage{makecell}
\usepackage{float}

\allowdisplaybreaks 

\usetikzlibrary{fit, angles, quotes,arrows, external}

\usepackage{xspace}

\newcommand{\abbr}[1]{\textsf{#1}\xspace}

\newcommand{\calA}{\mathcal{A}}
\newcommand{\calB}{\mathcal{B}}

\newcommand{\calF}{\mathcal{F}}

\newcommand{\calI}{\mathcal{I}}

\newcommand{\calK}{\mathcal{K}}
\newcommand{\calL}{\mathcal{L}}

\newcommand{\N}{\ensuremath\mathbb{N}}

\newcommand{\R}{\ensuremath\mathbb{R}}
\newcommand{\C}{\ensuremath\mathbb{C}}
\newcommand{\im}{\ensuremath{i}}

\newcommand{\diff}{\mathrm{d}}
\newcommand{\dr}{\,\diff r}
\newcommand{\ds}{\,\diff s}
\newcommand{\dt}{\,\diff t}

\newcommand{\dz}{\,\diff z}

\newcommand{\Linfty}{L^\infty}

\DeclareMathOperator{\range}{im}

\newcommand{\ISS}{\abbr{ISS}}

\newcommand{\PDEs}{\abbr{PDEs}}
\newcommand{\SOT}{\abbr{SOT}}

\newcommand{\NT}{\abbr{NT}}
\newcommand{\PINN}{\abbr{PINN}}
\newcommand{\PINNs}{\abbr{PINNs}}

\newcommand{\cs}{\Sigma}
\newcommand{\bcs}{\Sigma_\partial}

\newcommand{\inp}{u}
\newcommand{\inpdt}{\dot{u}}
\newcommand{\inpn}{u_n}

\newcommand{\state}{x}
\newcommand{\statedt}{\dot{\state}}
\newcommand{\stateinit}{\state_0}

\newcommand{\staten}{\state_n}
\newcommand{\statenidx}[1]{\state_{n, #1}}
\newcommand{\statendt}{\dot{\state}_n}

\newcommand{\stateCn}{\tilde{x}_n}
\newcommand{\stateDn}{\staten}
\newcommand{\stateEn}{\tilde{z}_n}
\newcommand{\stateFn}{z_n}

\newcommand{\stateCndt}{\dot{\tilde{x}}_n}
\newcommand{\stateDndt}{\dot{x}_n}
\newcommand{\stateEndt}{\dot{\tilde{z}}_n}
\newcommand{\stateFndt}{\dot{z}_n}

\newcommand{\coord}{\xi}
\newcommand{\deltacoord}{\Delta_{\coord, n}}
\newcommand{\spaceVar}{\xi}

\newcommand{\unitvector}[1]{\mathbf{e}_{#1}}

\newcommand{\diffconst}{\alpha}

\newcommand{\linop}[2]{\mathfrak{L}({#1},{#2})} 
\newcommand{\linopo}[1]{\mathfrak{L}({#1})} 
\newcommand{\ZU}{\Linfty(0,\infty; U)}
\newcommand{\ZUt}{\Linfty(0,t; U)}
\newcommand{\BIP}{\abbr{BIP}}
\newcommand{\Sect}{\abbr{Sect}}

\newcommand{\Xn}{X^n}

\newcommand{\tildeXn}{\tilde{X}^n}
\newcommand{\Xmin}{X_{-1}}
\newcommand{\Xminalpha}{X_{-1+\alpha}}
\newcommand{\Xalpha}{X_{\alpha}}
\newcommand{\InpSpace}{U}
\newcommand{\InpnSpace}{U_n}

\newcommand{\mexp}[0]{\mathrm{e}}
\newcommand{\timeInt}{\mathbb{T}}

\newcommand{\Bshort}{{\mathcal{B}}}
\newcommand{\Bnshort}{{\mathcal{B}}_n}

\newcommand{\bop}{\mathfrak{D}}
\newcommand{\boprinv}{\mathfrak{D}_0}

\newcommand{\bopn}{\mathfrak{D}_n}
\newcommand{\bopnrinv}{\mathfrak{D}_{n, 0}}
\newcommand{\bopnrinvU}{\mathfrak{D}^U_{n, 0}}

\newcommand{\Ainit}{\mathfrak{A}}
\newcommand{\Abase}{\mathcal{A}}
\newcommand{\Amin}{\mathcal{A}_{-1}}

\newcommand{\Aninit}{\mathfrak{A}_n}

\newcommand{\Anbase}{\mathcal{A}_n}

\newcommand{\Anmin}{\mathcal{A}_{n,-1}}

\newcommand{\res}[2]{R({#1}, {#2})}

\renewcommand{\S}{\mathcal{S}}
\newcommand{\Sn}{\mathcal{S}_n}
\newcommand{\Smin}{\mathcal{S}_{-1}}
\newcommand{\Snmin}{\mathcal{S}_{n,-1}}

\newcommand{\Id}{\mathcal{I}}

\newcommand{\IdU}{\mathcal{I}_{U}}

\newcommand{\En}{\mathcal{E}_n}
\newcommand{\Enmin}{\mathcal{E}_{n, -1}}
\newcommand{\tildeEn}{\tilde{\mathcal{E}}_n}
\newcommand{\mue}{\mu_\mathrm{e}}

\newcommand{\Pn}{\mathcal{P}_n}

\newcommand{\tildePn}{\tilde{\mathcal{P}}_n}
\newcommand{\mup}{\mu_\mathrm{p}}

\newcommand{\PnU}{\mathcal{P}^U_n}

\ifdefined\Rn
    \renewcommand{\Rn}{\mathcal{R}_n}
\else
    \newcommand{\Rn}{\mathcal{R}_n}
\fi

\newcommand{\Fn}{\mathcal{F}_n}

\newcommand{\descMatrix}{\mathcal{Q}_n}

\newcommand{\flowx}[3]{x({#1}; {#2}, {#3})}

\newcommand{\narrowinfty}{\xrightarrow[]{n\rightarrow\infty}}

\newcommand{\narrowinftySOT}[2]{\xrightarrow[\SOT(#1, #2)]{n\rightarrow\infty}}

\newcommand{\narrowinftyNT}[2]{\xrightarrow[\NT(#1, #2)]{n\rightarrow\infty}}

\newcommand{\sectorbound}{\varsigma}

\newcommand{\sectorboundmoved}{\alpha}

\newcommand{\othersector}{S}

\newcommand{\dlambda}{\, \mathrm{d}\lambda}
\renewcommand{\dr}{\, \mathrm{d}r}
\renewcommand{\path}{\Gamma}

\newcommand{\lineWidth}{1.2pt}

\newcommand{\sysA}{\abbr{(A)}}
\newcommand{\sysB}{\abbr{(B)}}
\newcommand{\sysC}{\abbr{(C)}}
\newcommand{\sysD}{\abbr{(D)}}

\newcommand{\sysF}{\abbr{(F)}}

\definecolor{kit-green}{RGB}{0, 150, 130}
\colorlet{kit-green100}{kit-green}
\colorlet{kit-green90}{kit-green!90!white}
\colorlet{kit-green80}{kit-green!80!white}
\colorlet{kit-green70}{kit-green!70!white}
\colorlet{kit-green60}{kit-green!60!white}
\colorlet{kit-green50}{kit-green!50!white}
\colorlet{kit-green40}{kit-green!40!white}
\colorlet{kit-green30}{kit-green!30!white}
\colorlet{kit-green25}{kit-green!25!white}
\colorlet{kit-green20}{kit-green!20!white}
\colorlet{kit-green15}{kit-green!15!white}
\colorlet{kit-green10}{kit-green!10!white}
\colorlet{kit-green5}{kit-green!5!white}

\definecolor{kit-blue}{RGB}{70, 100, 170}
\colorlet{kit-blue100}{kit-blue}
\colorlet{kit-blue90}{kit-blue!90!white}
\colorlet{kit-blue80}{kit-blue!80!white}
\colorlet{kit-blue70}{kit-blue!70!white}
\colorlet{kit-blue60}{kit-blue!60!white}
\colorlet{kit-blue50}{kit-blue!50!white}
\colorlet{kit-blue40}{kit-blue!40!white}
\colorlet{kit-blue30}{kit-blue!30!white}
\colorlet{kit-blue25}{kit-blue!25!white}
\colorlet{kit-blue20}{kit-blue!20!white}
\colorlet{kit-blue15}{kit-blue!15!white}
\colorlet{kit-blue10}{kit-blue!10!white}
\colorlet{kit-blue5}{kit-blue!5!white}

\definecolor{kit-royalblue}{RGB}{0, 45, 76}
\colorlet{kit-royalblue100}{kit-royalblue}
\colorlet{kit-royalblue90}{kit-royalblue!90!white}
\colorlet{kit-royalblue80}{kit-royalblue!80!white}
\colorlet{kit-royalblue70}{kit-royalblue!70!white}
\colorlet{kit-royalblue60}{kit-royalblue!60!white}
\colorlet{kit-royalblue50}{kit-royalblue!50!white}
\colorlet{kit-royalblue40}{kit-royalblue!40!white}
\colorlet{kit-royalblue30}{kit-royalblue!30!white}
\colorlet{kit-royalblue25}{kit-royalblue!25!white}
\colorlet{kit-royalblue20}{kit-royalblue!20!white}
\colorlet{kit-royalblue15}{kit-royalblue!15!white}
\colorlet{kit-royalblue10}{kit-royalblue!10!white}
\colorlet{kit-royalblue5}{kit-royalblue!5!white}

\definecolor{kit-iceblue100}{RGB}{30, 53, 69}
\definecolor{kit-iceblue70}{RGB}{68, 94, 111}
\definecolor{kit-iceblue50}{RGB}{168, 185, 196}
\definecolor{kit-iceblue30}{RGB}{218, 225, 230}

\definecolor{kit-red}{RGB}{162, 34, 35}
\colorlet{kit-red100}{kit-red}
\colorlet{kit-red90}{kit-red!90!white}
\colorlet{kit-red80}{kit-red!80!white}
\colorlet{kit-red70}{kit-red!70!white}
\colorlet{kit-red60}{kit-red!60!white}
\colorlet{kit-red50}{kit-red!50!white}
\colorlet{kit-red40}{kit-red!40!white}
\colorlet{kit-red30}{kit-red!30!white}
\colorlet{kit-red25}{kit-red!25!white}
\colorlet{kit-red20}{kit-red!20!white}
\colorlet{kit-red15}{kit-red!15!white}
\colorlet{kit-red10}{kit-red!10!white}
\colorlet{kit-red5}{kit-red!5!white}

\definecolor{kit-yellow}{RGB}{252, 229, 0}
\colorlet{kit-yellow100}{kit-yellow}
\colorlet{kit-yellow90}{kit-yellow!90!white}
\colorlet{kit-yellow80}{kit-yellow!80!white}
\colorlet{kit-yellow70}{kit-yellow!70!white}
\colorlet{kit-yellow60}{kit-yellow!60!white}
\colorlet{kit-yellow50}{kit-yellow!50!white}
\colorlet{kit-yellow40}{kit-yellow!40!white}
\colorlet{kit-yellow30}{kit-yellow!30!white}
\colorlet{kit-yellow25}{kit-yellow!25!white}
\colorlet{kit-yellow20}{kit-yellow!20!white}
\colorlet{kit-yellow15}{kit-yellow!15!white}
\colorlet{kit-yellow10}{kit-yellow!10!white}
\colorlet{kit-yellow5}{kit-yellow!5!white}

\definecolor{kit-orange}{RGB}{223, 155, 27}
\colorlet{kit-orange100}{kit-orange}
\colorlet{kit-orange90}{kit-orange!90!white}
\colorlet{kit-orange80}{kit-orange!80!white}
\colorlet{kit-orange70}{kit-orange!70!white}
\colorlet{kit-orange60}{kit-orange!60!white}
\colorlet{kit-orange50}{kit-orange!50!white}
\colorlet{kit-orange40}{kit-orange!40!white}
\colorlet{kit-orange30}{kit-orange!30!white}
\colorlet{kit-orange25}{kit-orange!25!white}
\colorlet{kit-orange20}{kit-orange!20!white}
\colorlet{kit-orange15}{kit-orange!15!white}
\colorlet{kit-orange10}{kit-orange!10!white}
\colorlet{kit-orange5}{kit-orange!5!white}

\definecolor{kit-lightgreen}{RGB}{140, 182, 60}
\colorlet{kit-lightgreen100}{kit-lightgreen}
\colorlet{kit-lightgreen90}{kit-lightgreen!90!white}
\colorlet{kit-lightgreen80}{kit-lightgreen!80!white}
\colorlet{kit-lightgreen70}{kit-lightgreen!70!white}
\colorlet{kit-lightgreen60}{kit-lightgreen!60!white}
\colorlet{kit-lightgreen50}{kit-lightgreen!50!white}
\colorlet{kit-lightgreen40}{kit-lightgreen!40!white}
\colorlet{kit-lightgreen30}{kit-lightgreen!30!white}
\colorlet{kit-lightgreen25}{kit-lightgreen!25!white}
\colorlet{kit-lightgreen20}{kit-lightgreen!20!white}
\colorlet{kit-lightgreen15}{kit-lightgreen!15!white}
\colorlet{kit-lightgreen10}{kit-lightgreen!10!white}
\colorlet{kit-lightgreen5}{kit-lightgreen!5!white}

\definecolor{kit-purple}{RGB}{163, 16, 124}
\colorlet{kit-purple100}{kit-purple}
\colorlet{kit-purple90}{kit-purple!90!white}
\colorlet{kit-purple80}{kit-purple!80!white}
\colorlet{kit-purple70}{kit-purple!70!white}
\colorlet{kit-purple60}{kit-purple!60!white}
\colorlet{kit-purple50}{kit-purple!50!white}
\colorlet{kit-purple40}{kit-purple!40!white}
\colorlet{kit-purple30}{kit-purple!30!white}
\colorlet{kit-purple25}{kit-purple!25!white}
\colorlet{kit-purple20}{kit-purple!20!white}
\colorlet{kit-purple15}{kit-purple!15!white}
\colorlet{kit-purple10}{kit-purple!10!white}
\colorlet{kit-purple5}{kit-purple!5!white}

\definecolor{kit-brown}{RGB}{167, 130, 46}
\colorlet{kit-brown100}{kit-brown}
\colorlet{kit-brown90}{kit-brown!90!white}
\colorlet{kit-brown80}{kit-brown!80!white}
\colorlet{kit-brown70}{kit-brown!70!white}
\colorlet{kit-brown60}{kit-brown!60!white}
\colorlet{kit-brown50}{kit-brown!50!white}
\colorlet{kit-brown40}{kit-brown!40!white}
\colorlet{kit-brown30}{kit-brown!30!white}
\colorlet{kit-brown25}{kit-brown!25!white}
\colorlet{kit-brown20}{kit-brown!20!white}
\colorlet{kit-brown15}{kit-brown!15!white}
\colorlet{kit-brown10}{kit-brown!10!white}
\colorlet{kit-brown5}{kit-brown!5!white}

\definecolor{kit-cyan}{RGB}{35, 161, 224}
\colorlet{kit-cyan100}{kit-cyan}
\colorlet{kit-cyan90}{kit-cyan!90!white}
\colorlet{kit-cyan80}{kit-cyan!80!white}
\colorlet{kit-cyan70}{kit-cyan!70!white}
\colorlet{kit-cyan60}{kit-cyan!60!white}
\colorlet{kit-cyan50}{kit-cyan!50!white}
\colorlet{kit-cyan40}{kit-cyan!40!white}
\colorlet{kit-cyan30}{kit-cyan!30!white}
\colorlet{kit-cyan25}{kit-cyan!25!white}
\colorlet{kit-cyan20}{kit-cyan!20!white}
\colorlet{kit-cyan15}{kit-cyan!15!white}
\colorlet{kit-cyan10}{kit-cyan!10!white}
\colorlet{kit-cyan5}{kit-cyan!5!white}

\definecolor{kit-gray}{RGB}{0, 0, 0}
\colorlet{kit-gray100}{kit-gray}
\colorlet{kit-gray90}{kit-gray!90!white}
\colorlet{kit-gray80}{kit-gray!80!white}
\colorlet{kit-gray70}{kit-gray!70!white}
\colorlet{kit-gray60}{kit-gray!60!white}
\colorlet{kit-gray50}{kit-gray!50!white}
\colorlet{kit-gray40}{kit-gray!40!white}
\colorlet{kit-gray30}{kit-gray!30!white}
\colorlet{kit-gray25}{kit-gray!25!white}
\colorlet{kit-gray20}{kit-gray!20!white}
\colorlet{kit-gray15}{kit-gray!15!white}
\colorlet{kit-gray10}{kit-gray!10!white}
\colorlet{kit-gray5}{kit-gray!5!white}

\title[Estimation of ISS gain functions from finite-dimensional approximations]{Estimation of the input-to-state stability gain functions from finite-dimensional approximations}

\author[B.~Hillebrecht, B.~Unger]{Birgit Hillebrecht${}^{\star,\dagger}$ \and Benjamin Unger${}^{\star,\ddagger}$}

\address{${}^{\star}$ Institute for Applied and Numerical Mathematics, Karlsruhe Institute of Technology, Karlsruhe, 76131, Germany.}

\address{${}^{\dagger}$ ORCID: 0000-0001-5361-0505}
\email{birgit.hillebrecht@kit.edu}

\address{${}^{\ddagger}$ ORCID: 0000-0003-4272-1079}
\email{benjamin.unger@kit.edu}

\begin{document}

\begin{abstract}
	Since the concept of input-to-state stability (ISS) was introduced, it has been extensively investigated for finite-dimensional control systems and has recently received attention for infinite-dimensional systems. While numerical techniques provide a bridge between these two worlds, a rigorous connection between the ISS of an infinite-dimensional system with an unbounded control operator and the properties of its finite-dimensional approximations has not yet been established. 
  In this manuscript, we make a first step towards closing this gap by investigating numerical approximations of linear (boundary) control systems using semigroup theory. Specifically, we focus on linear boundary control systems where the autonomous evolution is governed by an analytic semigroup. For these systems, we show that ISS gains can be computed from approximations. We illustrate the applicability of these findings using a one-dimensional heat equation with Dirichlet boundary control for which reference ISS gains are known.
\end{abstract}

\maketitle

%


\section{Introduction}
\label{sec:introduction}

Estimating system states in the presence of external inputs is a central problem in control theory, which motivated the concept of input-to-state stability (\ISS) \cite{Son89}. \ISS has been particularly well studied for finite-dimensional systems \cite{SonW95, Son08} while the interest for infinite-dimensional systems has grown recently \cite{JacNPS16,JacNPS18,MirP20,Sch20,KarK19}. The numerical analysis link between the two settings has been explored for bounded control operators \cite{BanW89}, but remains unstudied for unbounded linear control operators. Hence, approximating infinite-dimensional \ISS gain functions by means of finite-dimensional discretizations is still an open problem, in particular, when the control acts on the boundary only. This manuscript presents the first rigorous framework that connects numerical analysis and \ISS theory for a specific class of parabolic equations with unbounded control operators.

In more detail, we consider abstract infinite-dimensional control systems of the form 
\begin{equation}
    \label{eq:cs_base}
    \left\{\quad\begin{aligned}
        \statedt(t) &= \Abase \state(t) + \Bshort \inp(t), \\
        \state(0) & = \stateinit,
    \end{aligned}\right.
\end{equation}
with state $\state(t) \in X$ and input $\inp(t) \in \InpSpace$. The state space $X$ is a Banach space, the input space $U$ is a normed vector space, the operator $\Abase \colon D(\Abase)\subseteq X\longrightarrow X$ generates an analytic semigroup, and the (generally unbounded) operator $\calB\colon U\longrightarrow X$ models how the input impacts the state; the precise degree of unboundedness of \(\mathcal B\) will be specified in the forthcoming \Cref{subsec:control-systems}. Systems with unbounded control operators $\Bshort$ naturally appear when transforming partial differential equations (\PDEs) with boundary control, i.e., systems of the form (cf.~\cite[eq.~10.1]{CurZ20})
\begin{equation}
    \label{eq:bcs_base_intro}
    \left\{ \quad \begin{aligned}
        \statedt(t) &= \Ainit \state(t), \\
        \bop \state(t) &= \inp(t),\\
        \state(0)&= \stateinit,
    \end{aligned}\right.
\end{equation}
where $\bop$ is a suitable trace operator, into the form \eqref{eq:cs_base} \cite[Prop.\,2.8,~Part\,4]{Sch20}. In the following, we refer to~\eqref{eq:bcs_base_intro} as a \emph{boundary control system}. 

In this work, we study $\Linfty$-\ISS for both aforementioned systems~\eqref{eq:cs_base} and~\eqref{eq:bcs_base_intro}. Specifically, we seek to quantify functions $\beta$ and $\gamma$ in suitable function spaces (see the forthcoming \Cref{subsec:control-systems} for details), so-called \ISS gains, such that the state $x(t)$ satisfies the estimate
\begin{equation}
    \label{eq:iss}
    \|x (t)\|_X \le \beta(\|x_0\|_X, t) + \gamma (\|u\|_{\ZUt})
\end{equation}
for all $t\geq 0$.

Our interest in quantifying suitable \ISS gains is motivated by the posteriori error bound for physics-informed neural networks (\PINNs) derived by the authors in \cite{HilU25a}. In that setting, the input $\inp$ is not a control variable but is induced by the mismatch between the \PINN prediction and the prescribed boundary condition. This boundary error then has to be reflected in the overall prediction error of the \PINN, which can be bounded by the \ISS gain~$\gamma$. Consequently, the practical use of the error estimator hinges on the availability of a valid gain function~$\gamma$. At present, such gain functions are typically available only from rigorous analytical results for problems posed on simple geometries, such as intervals or rectangular domains, which renders the error bound inapplicable to many engineering problems. A first step toward overcoming this limitation was taken in~\cite{HilU25b}, where a computable bound resembling the \ISS estimate~\eqref{eq:iss} was introduced. However, this approach required the input~$\inp$, i.e., the boundary error, to be differentiable with respect to time. This paper eliminates the differentiability requirement and reinstates the original \ISS estimate, thereby broadening the applicability of the theory.

With this motivation in mind, we deliberately avoid relying on explicit knowledge of functional-analytic properties that are often difficult to obtain for problems defined on complicated geometries. These properties include, for example, the spectrum of the operator $\Abase$ or the extrapolated operator $\Amin$. Instead, we focus on strategies to retrieve $\beta$ and $\gamma$ via finite-dimensional approximations of the form
\begin{equation}\label{eq:cs_approx}
    \left\{\quad\begin{aligned}
        \statendt(t)& =\Anbase \staten(t) + \Bnshort \inp(t), \\
        \staten(0)& = \Pn \stateinit,
    \end{aligned}\right.
\end{equation}
with approximate state $\staten(t) \in \Xn$ evolves in a Banach space $\Xn$, $\Pn\colon X\to\Xn$ maps an element $\state \in X$ into the discretized space $\Xn$, and $\Bnshort$ models the impact of the input $\inp$ on the evolution of the state. 

Our main contributions are the following:
\begin{enumerate}
	\item We first focus on the estimation of \ISS gains for~\eqref{eq:cs_base} when the operator $\Bshort \colon \InpSpace \rightarrow X$ is known and bounded into the fractional power space $\Xminalpha$ for some $\alpha \in (0,1)$. We show in \Cref{thm:bound_cont_approx} that under the assumptions stated in \Cref{problem:cs}, the \ISS gains can be taken as suitable limits from the \ISS gains of the finite-dimensional approximation~\eqref{eq:cs_approx}.
	\item We then focus on the link between~\eqref{eq:bcs_base_intro} and~\eqref{eq:cs_base}, i.e., to settings where $\Bshort$ is not explicitly available; cf.~\Cref{problem:bcs}. By introducing a pre-boundary closure discretization, we extend the results from \Cref{thm:dist_contr} to boundary control approximations in \Cref{thm:bound_cont_approx}. We emphasize that for this step, we require the underlying Banach spaces to be complex.
\end{enumerate}
Finally, in \Cref{sec:numerics}, we apply these results to the one-dimensional heat equation with Dirichlet boundary control and compare the gain functions determined via finite-dimensional approximations with the theoretical gain functions available in the literature.

\subsection*{Organization of the manuscript}
After this introduction, we review concepts from functional analysis, abstract control systems, approximation theory in the sense of Trotter--Kato, and the precise problem statements in \Cref{sec:preliminaries}. The estimation of the \ISS gains for~\eqref{eq:cs_base} and~\eqref{eq:bcs_base_intro} are discussed in \Cref{sec:distributed-control-results,sec:boundary-control-results}, respectively, relying on a general known approximation result presented in the end of \Cref{sec:preliminaries}. We apply the theoretical results to a one-dimensional heat equation in~\Cref{sec:numerics} and conclude the paper with a summary and discussion in \Cref{sec:discussion}.

\subsection*{Notation}
We denote the space of bounded linear operators between  Banach spaces X and Y by $\linop{X}{Y}$ and write $\linopo{X} \coloneqq \linop{X}{X}$ if $Y = X$. For a linear operator $\calA \colon D(\calA) \subseteq X \rightarrow X$ we denote the range by $\range(\calA)$ and the graph norm by $\|\cdot \|_{\calA}$. Further, the resolvent set and the spectrum are denoted by $\rho(\calA)$ and $\sigma(\calA)$, respectively. The resolvent is written as $R(\lambda, \calA) \coloneqq (\lambda \Id - \calA)^{-1}$. In case $\calA$ generates a strongly continuous semigroup (short: $C_0$-semigroup) of operators $\{\S(t)\}_{t\ge 0}$, we say that the semigroup is of type $(M, \omega)$ if $\| \S(t) \|_X \le M \mexp^{\omega t}$. Convergence in the strong and norm topology for operators for a sequence of operators $\calK_n \in \linop{X}{Y}$ are denoted by 
\begin{equation*}
    \calK_n \narrowinftySOT{X}{Y} \calK \qquad \mathrm{and}\qquad  \calK_n \narrowinftyNT{X}{Y} \calK,
\end{equation*}
respectively. Further, given a normed space $X$, we denote the closure with respect to a norm $\|\cdot \|$ by $\overline{X}^{\|\cdot \|}$.
The $j$th unit vector in $\R^n$ is denoted by $\unitvector{j}$, the dimension $n\in \N$ is not stated explicitly as it will be clear from the context.
We use $\im$ to denote the imaginary unit, and for $z\in\C$, the argument is denoted by $\arg(z)$. With a slight misuse of notation, we use the symbol $\Gamma$ to denote both an abstract path and the Gamma function $\Gamma(z) = \int_0^\infty t^{z-1} \mathrm{e}^{-t} \dt$.

\section{Preliminaries and problem statements}
\label{sec:preliminaries}

In this section, we gather concepts and results from semigroup and control theory, which are needed throughout the manuscript. First, we recall essential facts about analytic semigroups in \Cref{subsec:sectorial-semigroups}. Then, in \Cref{subsec:control-systems}, we define the concrete control setting and revisit the precise definition of \ISS. Second, the approximation setting is summarized in \Cref{subsec:approximation-semigroups}. Given these prerequisites, we formulate the two central problem statements \Cref{problem:cs,problem:bcs} in \Cref{subsec:problem-statement}.

\subsection{Sectorial operators, analytic semigroups, and fractional powers}
\label{subsec:sectorial-semigroups}

Throughout the paper, we assume that the operator $-\Abase$ in \eqref{eq:cs_base} is sectorial of angle $0 < \sectorbound \le \tfrac{\pi}{2}$.

\begin{definition}[{\!\!\cite[Ch.\,1,\,§1]{Haa03}}]
    \label{def:sectorial-operator}
    A closed linear operator $-\calA$ is called \emph{sectorial} of angle $\sectorbound \in [0, \tfrac{\pi}{2})$ if 
    \begin{enumerate}
        \item the spectrum is contained in a sector of angle $\sectorbound$, i.e., 
        \begin{equation*}
        \sigma(-\Abase) \subseteq \othersector_{\sectorbound} \coloneqq \begin{cases} 
        	\lbrace z \in \C \mid |\mathrm{arg}\, z | < \sectorbound, \;z\neq 0\rbrace, & \text{if } \sectorbound > 0, \\
        (0, \infty), & \text{else},
        \end{cases}
        \end{equation*}
        \item and for all  $\sectorbound^\prime \in (\sectorbound, \pi)$ there exists $N > 0 $ such that $\|\lambda R(\lambda,-\calA)\| \le N$ in $\lambda \in \C\setminus  \overline{\othersector_{\sectorbound^\prime}}$.\label{item:resolvent-bound-1}
    \end{enumerate}
    The set of all sectorial operators of angle $\sectorbound$ is denoted by $\Sect(\sectorbound)$.
\end{definition} 

A graphical relation between the sector $\othersector_\sectorbound$ and the spectrum $\sigma(-\Abase)$ is depicted in \Cref{fig:sector-definition}.

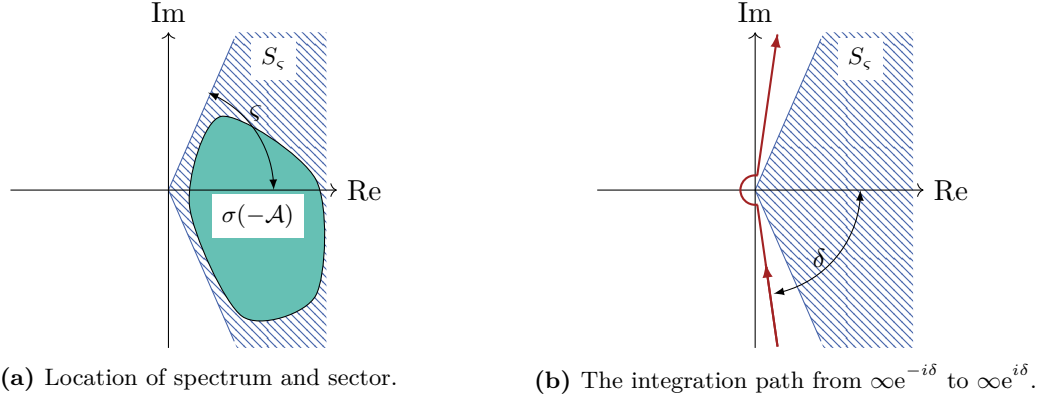
\begin{figure}[t]
    \centering
    \begin{subfigure}[t]{0.48\linewidth}
      \centering
        \begin{tikzpicture}[scale=0.70]
        \draw[scale=0.5, domain=0:2.5, smooth, variable=\x, kit-blue] plot ({\x}, {-2.3*\x});
        \draw[scale=0.5, domain=0:2.5, smooth, variable=\x, kit-blue] plot ({\x}, {2.3*\x});        
        \fill[pattern color=kit-blue, pattern=north west lines] (3,0)--(3, 3)--(1.3, 3) -- (0,0) --(1.3, -3) -- (3,-3);
        \filldraw[fill=kit-green60] plot[smooth cycle] coordinates {(2.8, .2) (1, 1.4) (0.4, -0.3) (1.4, -2.4) (2.8, -2)};
        \draw[latex-latex]  (68:2.0) arc(68:0:2.0) node[midway,above]{$\sectorbound$};
        \node[fill=white] at (2, 2.5) (S) {\footnotesize{$\othersector_{\sectorbound}$}};
        \node[fill=white] at (1.7, -0.5) (S) {\footnotesize{$\sigma(-\calA)$}};

        \draw[->] (-3, 0) -- (3.2, 0) node[right] {Re};
        \draw[->] (0, -3) -- (0, 3) node[above] {Im};
        
    \end{tikzpicture}
    \subcaption{Location of spectrum and sector.}
    \label{subfig:spectrum_sector}
  \end{subfigure}
  \hfill
  \begin{subfigure}[t]{0.48\linewidth}
    \centering
    \begin{tikzpicture}[scale=0.70]
        \draw[scale=0.5, domain=0:2.5, smooth, variable=\x, kit-blue] plot ({\x}, {-2.3*\x});
        \draw[scale=0.5, domain=0:2.5, smooth, variable=\x, kit-blue] plot ({\x}, {2.3*\x});        
        \fill[pattern color=kit-blue, pattern=north west lines] (3,0)--(3, 3)--(1.3, 3) -- (0,0) --(1.3, -3) -- (3,-3);
        \node[fill=white] at (2, 2.5) (S) {\footnotesize{$\othersector_{\sectorbound}$}};

        \draw[->] (-3, 0) -- (3.2, 0) node[right] {Re};
        \draw[->] (0, -3) -- (0, 3) node[above] {Im};
        
        \draw[scale=0.5, domain=-0.85:-0.4,thick,  smooth, variable=\x, kit-red, -{latex}] plot ({-\x}, {7*\x});
        \draw[scale=0.5, domain=-0.85:-0.08,thick,  smooth, variable=\x, kit-red] plot ({-\x}, {7*\x});
        \draw[scale=0.5, domain=0.85:0.08, thick, smooth, variable=\x, kit-red, latex-] plot ({\x}, {7*\x});
        \draw [thick, -, kit-red] (-0.28,0) arc [radius=-0.28,start angle=0,end angle=100];
        \draw [thick, -, kit-red] (-0.28,0) arc [radius=-0.28,start angle=0,end angle=-100];

        \draw[latex-latex]  (-80:2.0) arc(-80:0:2.0) node[midway, left]{\small{$\delta$}};
                                
    \end{tikzpicture}
    \subcaption{The integration path from $\infty \mexp^{-i\delta}$ to $\infty \mexp^{i\delta}$.}
    \label{subfig:intration_path_sector}
  \end{subfigure}
    \caption{Illustration of the spectrum $\sigma(-\calA)$ and the sector $\othersector_{\sectorbound}$ from the definition of a sectorial operator $-\calA$ and the integration path of the definition of fractional powers.}
    \label{fig:sector-definition}
\end{figure}

\begin{definition}[{\!\!\cite[Ch.\,1,\,§1]{Haa03}}]
    A family of operators $-\Anbase$ for $n\in \N$ is called \emph{uniformly sectorial of angle $\sectorbound \in [0, \tfrac{\pi}{2})$} if 
    \begin{enumerate}
        \item $-\Anbase \in \Sect(\sectorbound)$ for all $n\in\N$ , and
        \item for all $\sectorbound^\prime \in (\sectorbound, \pi)$ it holds that $\sup_{n\in\N} \sup_{\lambda \in \C \setminus \overline{\othersector_{\sectorbound^\prime}}} \|\lambda \res{\lambda}{-\Anbase}\|_{\linopo{X}} < \infty$.
    \end{enumerate}
\end{definition}

\begin{definition}[{\!\!\cite[Def.\,II.4.5]{EngN00}}]
    \label{def:analytic-semigroup}
    A semigroup of operators $\{\S(t)\}_{t\ge 0}$ is called \emph{analytic} of angle $\sectorbound \in (0, \tfrac{\pi}{2}]$ if for a sector $\othersector_{\sectorbound} \coloneqq \{ z\in \C \mid |\arg(z)| < \sectorbound,\; z\neq 0\}$ the following holds:
    \begin{enumerate}
        \item $\S(0) = \calI$ and $\S(z_1)\S(z_2) = \S(z_1+z_2)$ for $z_1, z_2 \in \othersector_{\sectorbound}$,
        \item the map $z \mapsto \S(z)$ is analytic in $\othersector_{\sectorbound }$, and 
        \item $\lim_{z \rightarrow 0, z\in \othersector_{\sectorbound -\epsilon}} \S(z) x = x$ for $\epsilon \in (0, \sectorbound)$ and all $x\in X$.
    \end{enumerate}
\end{definition}

Sectorial operators and analytic semigroups have many useful properties. Amongst them, the following three will prove especially advantageous in the subsequent discussion. First, the image of the semigroup  lies in the domain of the generator, i.e., $\range(\S(t)) \subseteq D(\Abase)$ for all $t > 0$ \cite[Ch. II, Thm. 4.6, part (c)]{EngN00}. Second, the product $\Abase \S(t)$ is bounded \cite[Ch. 2, proof of Thm. 5.2]{Paz89} by
\begin{equation}
    \label{eq:AS_bounded}
    \| \Abase \S(t)\|_{\linopo{X}} \le \frac{M}{\pi\cos(\theta)} \frac{1}{t}
\end{equation}
for $\theta \in (\tfrac{\pi}{2}, \sectorbound + \tfrac{\pi}{2})$, where $M$ denotes the first parameter from the growth bound of the semigroup $(M, \omega)$. Third, a particularly useful feature of sectorial operators is that fractional powers can be defined and computed using path integrals. A suitable example of a path $\Gamma$ is displayed in \Cref{subfig:intration_path_sector}.

\begin{definition}[{\!\!\cite[Def.\,II.5.25,\,II.5.31]{EngN00}}]
    \label{def:fractional-power}
    For $\alpha \in (0,1)$, the \emph{negative fractional power} of a sectorial operator $-\calA \in \Sect(\sectorbound)$ is defined as 
    \begin{equation*}
        (-\calA)^{-\alpha} \coloneqq \frac{1}{2 \pi i} \int_\path s^{-\alpha} \res{s}{-\calA} \ds
    \end{equation*}
    for a suitable path $\path$ which runs in $\C \setminus \overline{\othersector_\sectorbound}$ from $\infty \mexp^{-i\delta}$ to $\infty \mexp^{i\delta}$ for $\delta\in(\sectorbound, \pi)$.
    The positive fractional power is then defined as the inverse 
    \begin{equation*}
        (-\calA)^\alpha \coloneqq ((-\calA)^{-\alpha})^{-1}
    \end{equation*}
    on the domain $D((-\calA)^\alpha) = \range((-\calA)^{-\alpha})$.
\end{definition}

For $\alpha \in (0,1)$, the negative fractional power can also be represented as \cite[Ch.\,2, eq.\,(6.4)]{Paz89}
\begin{equation}
    \label{eq:fractional_power_line_integral}
    (-\calA)^{-\alpha} = \frac{\sin(\pi \alpha)}{\pi} \int_0^\infty t^{-\alpha}R(t, -\calA) \dt
\end{equation}
or by using the generated semigroup $\{\S(t)\}_{t\geq 0}$ \cite[Ch.\,2, eq.\,(6.9)]{Paz89} 
\begin{equation}
    \label{eq:fractional_power_integral_semigroup}
    (-\calA)^{-\alpha} = \frac{1}{\Gamma(\alpha)} \int_0^\infty t^{\alpha-1}\S(t)\dt.
\end{equation}
We remark that there are different representations of the fractional power in terms of the resolvent or the semigroup of operators;
a concise overview can be found in \cite{Haa06}.

\begin{definition}
    \label{def:fractional-power-spaces}
    For a sectorial operator $-\calA$ and $\alpha \in (-1, 1)$, we define the norm $\|\cdot \|_{\Xalpha} \coloneqq \| (-\calA)^{\alpha} \cdot \|_X$ and the  \emph{fractional power space} 
    \begin{equation*}
        \Xalpha \coloneqq (D((-\calA)^\alpha), \|\cdot \|_{\Xalpha}).
    \end{equation*}
\end{definition}

In the context of fractional power spaces, we use in \Cref{sec:boundary-control-results} the moment inequality and their identification with complex interpolation spaces (c.f.~\cite[Thm.~3]{See71}), which is on complex Banach spaces $X$ implied by $-\Abase$ having bounded imaginary powers. We briefly introduce all required concepts next and refer the reader for a concise introduction to interpolation theory to \cite[Ch.~1]{Lun12} and in \cite{Tri95}. 

\begin{lemma}[{\!\!\cite[Ch.\,II,~Thm.\,5.34]{EngN00}}]
    \label{lem:interpolation-fractional-power}
    Consider a sectorial operator $-\calA$. Then for $\alpha < \beta < \gamma$, there exists an $L > 0$ such that 
    \begin{equation*}
        \| (-\Abase)^\beta x \|_{X} \le L \| (-\Abase)^\alpha x\|_X^{\frac{\gamma- \beta}{\gamma-\alpha}}\| (-\Abase)^\gamma x\|_X^{\frac{\beta- \alpha}{\gamma-\alpha}}.
    \end{equation*}
\end{lemma}

\begin{definition}[{\!\!\cite[Def.\,1.1.1]{Lun12}}]
    Given Banach spaces $X, D, Y$ such that $D \subseteq Y \subseteq X$ and $\alpha \in (0,1)$, then $Y$ is called an \emph{interpolation space} of order $\alpha$ if there is a constant $c$ such that
    \begin{equation}\label{eq:interpolation_property}
        \| x\|_Y \le c \|x\|_X^{1-\alpha} \|x \|_D^{\alpha} \quad \mathrm{for\;all}\; x\in D.
    \end{equation}
\end{definition}

In particular, we denote the complex interpolation spaces between $X$ and $D$ by $[X,D]_{\alpha}$, for the precise definition of this class of interpolation spaces, we refer the reader to \cite[Ch.\,11]{MarS01}.

\begin{definition}[{\!\!\cite[Cor.\,2.32]{Haa03}}]
    \label{def:bip}
    An injective, sectorial operator $-\calA \in \Sect(\sectorbound)$ is said to have \emph{bounded imaginary powers}, denoted by $-\calA \in \BIP(X)$, if 
    \begin{enumerate}
        \item $\overline{D(-\calA) \cap \range(-\calA)} = X$, and 
        \item $(-\calA)^{\im s} \in \linopo{X}$ for all $s \in \R$.
    \end{enumerate}
\end{definition}

We infer from \cite[Thm.~11.6.1]{MarS01} that $-\calA \in \BIP(X)$ implies that the fractional power spaces $\Xalpha$ coincide with the complex interpolation spaces $[X, D(-\calA)]_\alpha$ for $\alpha \in (0,1)$ and similarily for negative fractional powers. 

\subsection{Control systems and input-to-state stability}
\label{subsec:control-systems}

Before heading to the concrete definition of the control systems with unbounded control operators, we recall the definition of the extrapolated space $\Xmin$ and the extrapolated semigroup $\{\Smin(t)\}_{t\ge 0}$. This provides a precise framework for handling unbounded operators. 

\begin{definition}[{\!\!\cite[Def.\,II.5.4, Thm.\,II.5.5]{EngN00}}]
    \label{def:extrapolation-space}
    For a given Banach space $X$ and a semigroup of operators $\{\S(t)\}_{t\ge 0}$, which is generated by $\Abase \colon D(\Abase) \subseteq X \mapsto X$  with $0 \in \rho(\Abase)$, we define the norm 
    \begin{equation*}
        \|x \|_{\Xmin} \coloneqq \| \Abase^{-1} x\|_{X}. 
    \end{equation*}
    The completion of X with respect to $\|\cdot \|_{\Xmin}$ is denoted by $\Xmin$ and called an \emph{extrapolation space}. The continuous extensions of $\{\S(t)\}_{t\ge 0}$ and $\Abase$ to $\Xmin$ are denoted by $\{\Smin(t)\}_{t\ge 0}$ and $\Amin$ respectively. 
\end{definition}

\begin{definition}[{\!\!\cite[Def.\,2.1, Ex.\,2.4]{Sch20}}]
    \label{def:control_system}
    Consider a Banach space $X$, a normed space $U$, a closed, densely defined, linear operator $\Abase \colon D(\Abase) \subseteq X \rightarrow X$ that generates a $C_0$-semigroup $\{\S(t)\}_{t\ge 0}$ on~$X$, and an operator $\Bshort \in \linop{U}{\Xmin}$. The linear time-invariant system
    \begin{equation}
        \label{eq:cs-base-2}
        \statedt(t) = \Abase \state(t) + \Bshort \inp(t)
    \end{equation}
    is called a \emph{control system}, denoted as $\cs(\Abase, \Bshort; X, U)$, with state $\state(t) \in X$ and input $\inp(t) \in \InpSpace$. If~$\Abase$ is analytic, then we call $\cs(\Abase, \Bshort; X, U)$ an \emph{analytic control system}.
\end{definition}

Since $\Bshort \in \linop{U}{\Xmin}$, it is not guaranteed that 
\begin{equation}
    \label{eq:mild-solution}
    \flowx{t}{\stateinit}{\inp} = \S(t) x_0 + \int_0^t \Smin(t-s) \Bshort \inp(s) \ds
\end{equation} 
is an element of $X$, but rather $\flowx{t}{\stateinit}{\inp}\in\Xmin$. However, if \eqref{eq:mild-solution} lies in $X$, it is called the \emph{mild solution} of \eqref{eq:cs-base-2} \cite[Def.\,4.7]{MirP20}. The notion of \emph{admissibility} formalizes the existence of the mild solution for a set of allowed inputs.

\begin{definition}[{\!\!\cite{Wei89}}]
    \label{def:admissibility}
    The control system $\Sigma(\Abase, \Bshort; X, U)$ is called \emph{$\Linfty$-admissible}, if
    \begin{equation}
        \label{eq:admissibility}
        \int_0^t \Smin(t-s) \Bshort \inp(s) \ds \in X
    \end{equation}
    for all $t> 0$ and $\inp \in \ZU$.
\end{definition}

We now turn our attention to boundary control systems.

\begin{definition}[{\!\!\!\cite[Def.\,2.5]{Sch20}}]
    \label{def:boundary_control_ssystem}
    Given a Banach space $X$, a normed space $U$, a linear operator $\Ainit \colon D(\Ainit)\subseteq X \rightarrow X$, and a boundary operator $\bop \in \linop{D(\Ainit)}{U}$. Assume that $\Abase \coloneqq \Ainit|_{\ker\bop}$ is the generator of a $C_0$-semigroup $\{\S(t)\}_{t\ge 0}$ on $X$ and that $\bop$ has a linear right-inverse denoted by $\boprinv$, which is bounded from $U$ to $(D(\Ainit), \|\cdot\|_\Ainit)$. Then the system
    \begin{equation}
        \label{eq:bcs_base}
        \left\{\quad\begin{aligned}
            \statedt(t, \cdot ) &= \Ainit \state(t, \cdot) \quad& \mathrm{in}\; \timeInt \times  \Omega, \\
            \bop \state(t, \cdot) &= \inp(t) \quad& \mathrm{in}\; \timeInt \times \partial \Omega, \\
            \state(0, \cdot) & = \stateinit & \mathrm{in}\; \Omega,
        \end{aligned}\right.
    \end{equation}
    is called a \emph{boundary control system} and denoted with $\bcs(\Ainit, \bop; X, \InpSpace)$.
\end{definition}

As previously mentioned, control systems \eqref{eq:cs-base-2} with unbounded control operators $\Bshort$ naturally arise when transforming boundary control systems, which is referred to as the Fattorini trick~\cite{Fat68,Sch20}.

\begin{lemma}[{\!\!\cite{Fat68}, \cite[p.\,93]{Sch20}}]
    \label{lem:fattorini-trick}
    A boundary control system $\bcs(\Ainit, \bop; X, \InpSpace)$ can be transformed into an equivalent control system $\cs(\Ainit\mid_{\ker \bop}, \Bshort; X, \InpSpace)$ with 
    \begin{equation*}
        \Bshort \coloneqq \Ainit \boprinv - \Amin \boprinv \in \linop{\InpSpace}{\Xmin}.
    \end{equation*}
\end{lemma}

Finally, we recall the definition of \ISS for (boundary) control systems. This definition relies on the standard classes of comparison functions $\calK, \calL$, and $\calK\calL$, c.f. \cite{Kel14}, given as
\begin{equation*}
    \begin{aligned}
        \calK &\coloneqq  \left\{ \mu: \R_+ \rightarrow \R_{+} \mid \mu(0)=0, \mu\;\mathrm{cont.}, \mathrm{strictly\;increasing}\right\}, \\
        \calL &\coloneqq  \left\{  \mu: \R_{+}\rightarrow \R_{+} \mid\mu\;\mathrm{cont.}, \mathrm{strictly\;decreasing}, \lim_{t \rightarrow \infty} \mu(t)=0\right\},\,\mathrm{and} \\
        \calK\calL &\coloneqq  \left\{ \mu \colon \R_{+}^2 \rightarrow \R_{+} \mid \mu(\cdot, t) \in \calK \text{ for all } t\ge 0; \;  \mu(s, \cdot )\in \calL \text{ for all } s \ge 0\right\}. \\
    \end{aligned}
\end{equation*}

\begin{definition}[{\!\!\cite{Son08}}]
    \label{def:iss}
    An $\Linfty$-admissible control system $\cs( \Abase, \Bshort; X, U)$ is called \emph{$\Linfty$-\ISS} if there are functions $\beta \in \calK \calL, \; \gamma \in \calK$ such that the mild solution \eqref{eq:mild-solution} satisfies
    \begin{equation}
        \label{eq:iss-definition}
        \| \flowx{t}{\stateinit}{\inp}  \|_X \le \beta(\|\stateinit\|_X, t) + \gamma(\|\inp\|_{\ZU})
    \end{equation}
    for all $t\in (0, \infty)$ and $\inp\in \ZU$. Similarly, a boundary control system $\bcs(\Ainit, \bop; X, \InpSpace)$ is called~$\Linfty$-\ISS if the equivalent control system $\cs(\Ainit\vert_{\ker \bop}, \Bshort; X, \InpSpace)$ (in the sense of \Cref{lem:fattorini-trick}) is $\Linfty$-\ISS. The functions $\beta$ and $\gamma$ are called \emph{\ISS gains}.
\end{definition}

\subsection{Approximation of strongly continuous semigroups of operators}
\label{subsec:approximation-semigroups}

For the approximations \eqref{eq:cs_approx}, we adopt the framework from \cite{Nam23, Kur70} with the following definitions and assumption. 

\begin{definition}
    \label{def:convergence-banach-spaces}
    Given a Banach space $X$, consider a sequence of triples $(\Xn, \Pn, \En)_{n\in \N}$, each consisting of a Banach space $\Xn$, a bounded linear operator $\Pn \colon X \rightarrow \Xn$, and a linear right-inverse to $\Pn$ denoted by $\En \colon \Xn \rightarrow X$. If further 
    \begin{equation*}
        \lim_{n\rightarrow \infty} \| \Pn f\|_{\Xn} = \| f\|_X.
    \end{equation*}
    for all $f\in X$, then we call $(\Xn, \Pn, \En)_{n\in \N}$ an \emph{approximation sequence} for $X$. 
\end{definition}

For brevity, we write in the following only $(\Xn)_{n\in \N}$ if the restriction and extension operators $\Pn$ and $\En$ are not explicitely needed for the understanding. 

\begin{assumption}[{\!\!\cite[(A1)-(A3)]{ItoK98}}]
    \label{ass:approximating-banach-spaces}
    Let a Banach space $X$ and an approximation sequence $(\Xn, \Pn, \En)_{n\in \N}$ for $X$ be given. The operators $\Pn$ and $\En$ satisfy
    \begin{enumerate}[label=(TK\arabic*)]
        \item $\|\Pn\|_{\linop{X}{\Xn}}\le \mup$\,and $\|\En \|_{\linop{\Xn}{X}}\le \mue$\,for all $n\in\N$ with $\mup, \mue$\,independent of $n$,
        \item $\lim_{n\rightarrow \infty}\|\En \Pn \state - \state \|_X= 0$ for all $\state \in X$. \label{ass:EnPn-to-Id}
    \end{enumerate}
\end{assumption}

\begin{definition} 
    \label{def:convergence-semigroups}
    Consider a Banach space $X$, a semigroup of operators $\{\S(t)\}_{t\ge 0}$ on~$X$, an approximation sequence $(\Xn)_{n \in \N}$ for $X$, and for each $n\in \N$ a semigroup of operators $\{\Sn(t)\}_{t\ge 0}\in \linopo{\Xn}$. The sequence of semigroups $(\{\Sn(t)\}_{t\ge 0})_{n\in\N}$ is called \emph{convergent} to $\{\S(t)\}_{t\ge 0}$ if for all $\stateDn \in \Xn$ and $x\in X$ with $\lim_{n \rightarrow \infty}\En\stateDn = x$ 
    it holds that
    \begin{equation*}
        \lim_{n\rightarrow\infty} \En\Sn(t) \stateDn = \S(t) x \quad \text{for all}\;\; t\ge0.
    \end{equation*}
\end{definition}
This implies in particular that  
\begin{equation*}
    \En \Sn (t) \Pn \narrowinftySOT{X}{X} S(t) \quad \text{for all}\;\; t\ge0,
\end{equation*}
which we abbreviate as $\Sn(t) \narrowinftySOT{X}{X} \S(t)$.

\subsection{Problem statements}
\label{subsec:problem-statement}

The analysis of \ISS for abstract control systems and boundary control systems via finite-dimensional approximations leads to two closely related problems. In the first step, we study the abstract control system~\eqref{eq:cs-base-2} and its discretizations, with the goal of deriving explicit, computable \ISS gain functions. The precise problem formulation, including all relevant assumptions, is detailed in \Cref{problem:cs}. In the second step, addressed in \Cref{problem:bcs}, we extend the problem formulation to boundary control system~\eqref{eq:bcs_base} by exploiting the fact that the boundary operator~$\Bshort$ is realized via the Fattorini trick presented in \Cref{lem:fattorini-trick}.

\begin{problem} 
    \label{problem:cs}
    Consider an $\Linfty$-admissible control system $\cs(\Abase, \Bshort; X, U)$, a finite-dimensional approximation sequence $(\Xn)_{n\in\N}$, and control systems $\cs(\Anbase, \Bnshort; \Xn, U)$. Assuming that
    \begin{enumerate}[label=(A\arabic*)]
        \item \label{ass:cs:approximating-banach-spaces} the sequence of Banach spaces $(\Xn)_{n\in\N}$ fulfills \Cref{ass:approximating-banach-spaces}, 
        \item the operators $-\Abase$ and $-\Anbase$ for $n \in \N$ are uniformly sectorial of angle $\sectorbound \in [0, \tfrac{\pi}{2})$,\label{ass:uniform-sectoriality}
        \item the sequence of semigroups $\{\Sn(t)\}_{t\ge 0}$ generated by $\Anbase$ converges against $\{\S(t)\}_{t\ge 0}$ generated by $\Abase$, \label{ass::convergence-semigroups}
        \item there exists a path $\hat{\path}$ in $\C\setminus\overline{\othersector_{\sectorbound}}$ from $\infty\mexp^{-i\delta}$ to $\infty \mexp^{i\delta}$ for $\delta\in (\sectorbound,\pi)$ such that \label{ass:norm-convergence-resolvent-path} 
        \begin{equation*}
            \sup_{\lambda \in \hat{\path}}\| \En \res{\lambda}{-\Anbase} \Pn - \res{\lambda}{-\Abase}\|_X \narrowinfty 0 ,
        \end{equation*}
        \item and there exists an $\alpha \in (0,1)$ such that the control operators $\Bnshort$ approximate $\Bshort$ in the sense that $\En \Bnshort \narrowinftySOT{U}{\Xminalpha} \Bshort$, \label{ass:convergence-control-operators}
    \end{enumerate} 
    find \ISS gains for $\cs(\Abase, \Bshort; X, \InpSpace)$ from the sequence $\cs(\Anbase, \Bnshort; \Xn, \InpSpace)$.
\end{problem}

The following remarks are in order:
\begin{enumerate}
	\item The uniform sectoriality in \ref{ass:uniform-sectoriality} implies that for a fixed  $\sectorbound^\prime \in (\sectorbound, \pi)$ there exists a $D_n$ such that $\|\res{\lambda}{-\Anbase}\|_{\linopo{\Xn}} \le \tfrac{D_n}{|\lambda|+1}$ for all $\lambda \in \C\setminus\overline{\othersector_{\sectorbound^\prime}}$ and all $n\in \N$. Further, it also follows that $\sup_{n\in \N} D_n < \infty$.
	\item The Trotter--Kato theorem \cite{ItoK98} together with \ref{ass::convergence-semigroups} implies the pointwise convergence of the resolvents.
	\item Since $\Xn$ is finite-dimensional, $\Bnshort$ and $\En \Bnshort$ are bounded operators. Hence, $\Linfty$-admissibility of $\cs(\Abase, \En\Bnshort; X, \InpSpace)$ and $\cs(\Abase, \Bshort - \En\Bnshort; X, \InpSpace)$ is a consequence of the $\Linfty$-admissibility of $\cs(\Abase, \Bshort; X, \InpSpace)$.  
\end{enumerate}

To extend the problem setting to the case of boundary control systems, we replace \ref{ass:convergence-control-operators} by four other assumptions formulated in terms of the boundary control system.

\begin{problem}
    \label{problem:bcs}
    Given a complex Banach space X and a boundary control system $\bcs(\Ainit, \bop; X, U)$ for which
    \begin{enumerate}[label=(B\arabic*)]
        \item $\Abase \coloneqq \Ainit |_{\ker{\bop}}$ generates an analytic semigroup $\{\S(t)\}_{t\ge 0}$, \label{ass:bcs:generator}
        \item $\bop$ has a bounded, linear right inverse $\boprinv \in \linop{U}{D(\Ainit)}$, i.e., $\bop \boprinv = \IdU$, and there exists $\alpha\in(0,1)$ such that $\bop$ is bounded from $U$ to $\Xalpha$,\label{ass:bcs:bop-right-invertible}
        \item and the sectorial operator has bounded imaginary powers, i.e., $-\Abase \in \BIP(X)$,\label{ass:bcs:bip} 
    \end{enumerate}
    assume further, that a finite-dimensional approximation sequence $(\Xn)_{n \in \N}$, and control systems $\cs(\Anbase, \Bnshort; \Xn, U)$  fulfill \ref{ass:cs:approximating-banach-spaces}--\ref{ass:norm-convergence-resolvent-path} and
    \begin{enumerate}[label=(B\arabic*)]
      \addtocounter{enumi}{3}
      \item the approximations are consistent in the sense that $\En \Anbase \Pn$ is bounded uniformly in $n\in\N$ as both, operator in $\linop{D(\Abase)}{X}$ and $\linop{X}{\Xmin}$, \label{ass:consistent_approx}
    \end{enumerate}
    determine the \ISS gain $\gamma$ for $\bcs(\Ainit, \bop; X, U)$ such that neither an explicit construction of the extrapolated operator $\Amin$ nor access to the operator $\Bshort$ is required.
\end{problem}

\begin{remark}
  We emphasize that~\ref{ass:bcs:bip} can be replaced by any assumption that allows for the identification of the fractional power spaces with real or complex interpolation spaces. 
\end{remark}

To address \Cref{problem:cs,problem:bcs}, we first observe 
\begin{equation*}
    \begin{aligned}
        \| \flowx{t}{\stateinit}{\inp}  \|_X &\le \| \flowx{t}{\stateinit}{0} \|_X + \left\| \flowx{t}{0}{u}\right\|_X = \|\S(t) \stateinit\|_X + \left\| \int_0^t \Smin(t-s) \Bshort \inp(s) \ds \right\|_X.
    \end{aligned}
\end{equation*}
Hence, if $\beta \in \calK \calL$ and $\gamma\in \calK$ satisfy
\begin{align*}
    \beta(\|\stateinit\|, t) \ge \|\S(t) \stateinit\|_X, \quad \gamma(\|\inp\|_{\ZU}) \ge \left\| \int_0^t \Smin(t-s) \Bshort \inp(s) \ds \right\|_X
\end{align*}
they also satisfy \eqref{eq:iss-definition}. 
If the semigroup is of type $(M,\omega)$, then it is  clear that a valid choice of $\beta$ is 
\begin{equation*}
    \beta(s, t) \coloneqq M \mexp^{\omega t} s
\end{equation*}
for which the following approximation result for $\beta$ was shown in \cite{HilU25b}.

\begin{lemma}[{\!\!\cite[Thm.\,3.4]{HilU25b}}]
    \label{lem:omega_conv}
    Consider a Banach space $X$ and a sequence of Banach spaces $(\Xn)_{n \in \N}$ of $X$ which fulfills \Cref{ass:approximating-banach-spaces}, a semigroup of operators $\{\S(t)\}_{t \ge 0}$ on $X$, and a sequence of semigroups of operators $(\{\Sn(t)\}_{t\geq 0})_{n \in \N}$ on $\Xn$ of type $(M_n, \omega_n)$ respectively. Further, assume that
    \begin{itemize}
        \item[(i)] $\Sn(t) \narrowinftySOT{X}{X} \S(t)$ (in the sense of \Cref{def:convergence-semigroups}),
        \item[(ii)] $(M_n)_{n\in \N}$ and $(\omega_{n})_{n\in \N}$ are Cauchy sequences.
    \end{itemize}
    Then, a growth bound of the limit semigroup is given by $\|\S(t)\|_{\linopo{X}} \le M^\ast \mexp^{\omega^\ast t}$ with
    \begin{equation}
        \label{eq:growth-bound-limits}
        \omega^\ast \coloneqq \lim_{n\rightarrow\infty} \omega_{n} , \qquad
        M^\ast \coloneqq  \mup \mue \lim_{n \rightarrow \infty} M_n.
    \end{equation}
\end{lemma}

Building upon these results, we can now restrict the discussion to finding the \ISS gain $\gamma$ for settings as described in \Cref{problem:cs,problem:bcs}

\section{Approximation of the \ISS gain $\gamma$ for control systems}
\label{sec:distributed-control-results}

To solve \Cref{problem:cs}, it remains to find the \ISS gain $\gamma$ from the finite-dimensional control systems $\cs(\Anbase, \Bnshort; \Xn, U)$. Because $\S(t)$ is analytic and \ref{ass:convergence-control-operators} is assumed, we can estimate according to \cite[Prop.\,2.13]{Sch20}
\begin{align*}
	\left\| \int_0^t \Smin(t-s) \Bshort \inp(s) \ds \right\|_X &= \left\| \int_0^t \Smin (t-s)(-\Abase)^{1-\alpha} (-\Abase)^{-1+\alpha} \Bshort u(s)\ds \right\| _{X}\\
	&\le \int_0^t \left\| (-\Abase)^{1-\alpha} \S (t-s)\right\|_{\linopo{X}} \left\| (-\Abase)^{-1+\alpha} \Bshort u(s)\right\|_{X}\ds.
\end{align*}
Hence, we can derive a computable bound from the approximations by treating the two factors in the integrand separately. To obtain a bound for the first factor, we proceed by 
\begin{enumerate}
    \item finding a bound on $\| (-\Abase)^{1-\alpha}\Smin (t-s)\|_{\linopo{X}}$ which depends on the type~$(M, -\omega)$ of $\S(t)$ and~$D$ such that $\|R(\lambda, -\Abase)\| \le D (|\lambda|+1)^{-1}$ for $\lambda \in \C \setminus \overline{\othersector_{\sectorbound^\prime}}$ and $\sectorbound^\prime \in (\sectorbound, \pi)$ (see \Cref{lem:growth_bound_frac_power}),
    \item estimate $D$ from approximations (see \Cref{lem:estimate_D}), and to
    \item estimate $(M, -\omega)$ from approximations (see \Cref{lem:omega_conv}).   
\end{enumerate}
To estimate the second factor, we 
\begin{enumerate}
    \item apply the definition of the operator norm and bound 
    \begin{equation*}\left\| (-\Abase)^{-1+\alpha} \Bshort u(s)\right\|_{X} \le \left\| (-\Abase)^{-1+\alpha} \Bshort \right\|_{\linop{U}{X}}\left\| u(s)\right\|_{U},
    \end{equation*}
    \item show convergence of the negative fractional powers $(-\Anbase)^{-1+\alpha}$ to $(-\Abase)^{-1+\alpha}$ (see \Cref{lem:frac_power_neg}),
    \item and bound $\|(-\Abase)^{-1+\alpha}\Bshort\|_{\linop{U}{X}}$ by $\lim_{n\rightarrow\infty}\|(-\Anbase)^{-1+\alpha}\En\Bnshort\|_{\linop{U}{\Xn}}$ (see \Cref{lem:EnBn_in_Xmin}).
\end{enumerate}
Combining these results in the forthcoming \Cref{thm:dist_contr}, yields an \ISS gain $\gamma\in\calK$, which is computable from finite-dimensional approximations.

\begin{lemma} 
    \label{lem:growth_bound_frac_power}
    Let $\Abase$ be the generator of an analytic semigroup $\{\S(t)\}_{t\ge 0}$ on $X$ of type $(M, -\omega)$ and angle $\sectorbound$. Let $\sectorbound^\prime \in (\sectorbound, \pi)$ and $D > 0$ such that
    \begin{equation*}
        \|\res{\lambda}{-\Abase} \|_X \le \frac{D}{|\lambda| + 1} \; \text{for} \, \lambda \in \C \setminus \overline{\othersector_{\sectorbound^\prime}}
    \end{equation*}
    be given. For $\alpha \in (0, 1)$ and $\theta \in (\tfrac{\pi}{2},\tfrac{\pi}{2}+ \sectorbound)$ define
    \begin{align*}
        K_1(\omega,M,\alpha) &\coloneqq \frac{ \omega M }{\Gamma(1-\alpha)} \int_0^\infty s^{-\alpha} \mexp^{-\omega s}\ds,\,\mathrm{and}\\
        K_2(\omega,M,\alpha,D,\theta) &\coloneqq \frac{D }{\Gamma(1-\alpha)\pi |\cos(\theta)|} \int_0^\infty s^{-\alpha}  (1+s)^{-1} \ds.
    \end{align*}
    Then,
    \begin{equation*}
        \| (-\Abase)^\alpha \S(t) \|_{\linopo{X}} \le  K_1(\omega, M, \alpha) \mexp^{-\omega t}+ K_2(\omega, M, \alpha, D, \theta) \mexp^{-\omega t} t^{-\alpha}.
    \end{equation*}
\end{lemma}
\begin{proof}
    We apply a suitable adaptation of \cite[proof of Ch. 2, Thm. 5.2, eq. 5.13]{Paz89}, i.e., defining $\hat{\Abase} \coloneqq \Abase + \omega I $ with associated semigroup $\hat{\S}(t) = \mexp^{-\omega t} \S(t)$ of type $(M, 0)$ to find
    \begin{align*}
        \|\Abase \S(t)\|_{\linopo{X}} &= \|(\hat{\Abase} - \omega I) \mexp^{-\omega t} \hat{\S}(t) \|_{\linopo{X}} \le  \mexp^{-\omega t} \|\hat{\Abase} \hat{\S}(t)\|_{\linopo{X}} + \omega  \mexp^{-\omega t} \|\hat{\S}(t)\|_{\linopo{X}} \\
        ^{\eqref{eq:AS_bounded}}& \le \mexp^{-\omega t} \frac{D}{\pi |\cos(\theta)|} t^{-1} + \mexp^{-\omega t} \omega M = \mexp^{-\omega t} \left(\omega M  + \frac{D}{\pi |\cos(\theta)|} t^{-1}\right)
    \end{align*}
    and (using the alternative representation of the negative fractional power~\eqref{eq:fractional_power_integral_semigroup})
    \begin{align*}
        \| (-\Abase)^\alpha \S(t) \|_{\linopo{X}} &= \|(-\Abase)^{\alpha-1} (-\Abase) \S(t)\|_{\linopo{X}}\\
        &= \left\| \frac{1}{\Gamma(1-\alpha)} \int_0^\infty r^{-\alpha}(-\Abase)\S(r+t) \dr\right\|_{\linopo{X}}\\ 
        &\le \frac{1}{\Gamma(1-\alpha)} \int_0^\infty r^{-\alpha} \|\Abase\S(r+t)\|_{\linopo{X}} \dr\\ 
        &\le \frac{1}{\Gamma(1-\alpha)} \int_0^\infty r^{-\alpha} \mexp^{-\omega(t+r)}\left[\omega M +\frac{D}{\pi |\cos(\theta)|} (t+r)^{-1}\right] \dr \\
        &\overset{(\ast)}{\le} \frac{ \omega M\mexp^{-\omega t} }{\Gamma(1-\alpha)} \int_0^\infty r^{-\alpha} \mexp^{-\omega r}\dr +\frac{D\mexp^{-\omega t}}{\Gamma(1-\alpha)\pi |\cos(\theta)|} \int_0^\infty r^{-\alpha} (t+r)^{-1} \dr \\
        &=\frac{ \omega M\mexp^{-\omega t} }{\Gamma(1-\alpha)} \int_0^\infty r^{-\alpha} \mexp^{-\omega r}\dr + \frac{D t^{-\alpha} \mexp^{-\omega t}}{\Gamma(1-\alpha)\pi |\cos(\theta)|} \int_0^\infty s^{-\alpha} (1+s)^{-1} \ds 
    \end{align*}
    where $(\ast)$ holds because $\mexp^{\omega r} \le 1$ for $r\ge 0$ and $\omega <0$. This shows the claim.
\end{proof}

\begin{remark}
    In the literature, e.g., in~\cite[Ch. 2, Thm. 6.13, Part (c)]{Paz89}, the contribution with factor $K_1$ is absorbed into the term scaled by $K_2$ by suitably increasing both $K_2$ and the exponential growth rate. This yields a bound of the form 
    \begin{equation*}
        \| (-\Abase)^\alpha \S(t) \|_{\linopo{X}} \le  C  t^{-\alpha} \mexp^{-\tilde{\omega} t}
    \end{equation*}
    with suitable $C>0$ and $\tilde{\omega} > \omega$. In principle, all the following computations can be performed using this bound as well. Nevertheless, as we will approximate $\omega$, we do not adopt this simplification here to maintain computability.
\end{remark}

\begin{lemma}
    \label{lem:estimate_D}
    Let the setting be as described in \Cref{problem:cs} and $\sectorbound^\prime \in (\sectorbound, \pi)$. Then, 
    \begin{equation*}
        \|\res{\lambda}{-\Abase}\|_{\linopo{X}} \le \frac{\hat{D}}{|\lambda| + 1} \quad \text{with} \quad D = \mue\mup \sup_{n\in \N} D_n
    \end{equation*}
    for $\lambda \in \C \setminus \overline{\othersector_{\sectorbound^\prime}}$.
\end{lemma}
\begin{proof}
    Since by \ref{ass::convergence-semigroups} implies pointwise convergence of the resolvents in $X$, we can apply the Banach--Steinhaus theorem to $\En(|\lambda|+1)\res{\lambda}{-\Anbase}\Pn$ and conclude that for $\lambda \in \C \setminus \overline{\othersector_{\sectorbound^\prime}}$
    \begin{align*}
    	\|(|\lambda| +1)\res{\lambda}{-\Abase}\|_{\linopo{X}}
        &= \sup_{n\rightarrow\infty}\|\En(|\lambda| +1)\res{\lambda}{-\Anbase}\Pn\|_{\linopo{X}} \\
        & \le \mue\mup \sup_{n\rightarrow\infty} D_n = \hat{D}. \qedhere
    \end{align*}
\end{proof}
A similar result as the following is given in \cite[Thm.~2.4]{FujSS01}, where it is formulated quantitatively for finite element discretizations.
\begin{lemma} 
    \label{lem:frac_power_neg}
    Let the setting be as described in \Cref{problem:cs}. For $\alpha \in (0,1)$, the negative fractional powers converge in the norm topology
    \begin{equation*}
        (-\Anbase)^{-\alpha} \narrowinftyNT{X}{X} (-\Abase)^{-\alpha} .
    \end{equation*}
\end{lemma}
\begin{proof}
    Starting from \eqref{eq:fractional_power_line_integral}, we obtain for $x \in X$ and $\hat{\path}$ according to \ref{ass:norm-convergence-resolvent-path}
    \begin{align*}
        & \phantom{=} \left\|(-\Abase)^{-\alpha} - \En (-\Anbase)^{-\alpha} \Pn \right\|_{\linopo{X}} \le  \frac{1}{2\pi \im} \int_{\hat{\path}} \lambda^{-\alpha} \left\|R(\lambda,-\Abase) - \En R(\lambda, -\Anbase) \Pn\right\|_{\linopo{X}}  \dlambda
    \end{align*}
    with in which the integrand is bounded for $\lambda \in \hat{\path}$ and for all $n\in \N$ by
    \begin{align*}
        \lambda^{-\alpha} \left\|R(\lambda, -\Abase) - \En R(\lambda, -\Anbase) \Pn\right\|_{\linopo{X}} &\le  \lambda^{-\alpha} \biggl [ \left\|R(\lambda, -\Abase) \right\|_{\linopo{X}} + \left\| \En R(\lambda, -\Anbase) \Pn \right\|_{\linopo{X}}\biggr ] \\
        &\le  \lambda^{-\alpha} \left[ \frac{D}{\lambda +1 } + \frac{\mue\mup\sup_{n\in \N}D_n}{\lambda + 1}\right] \\
        & = \frac{(D+ \mue\mup\sup_{n\in\N}D_n)\lambda^{-\alpha}}{\lambda + 1} \eqcolon g_\alpha(\lambda).
    \end{align*} 
    Since $g_\alpha(\lambda)$ is integrable from 0 to $\infty$ for $\alpha \in (0,1)$, the dominated convergence theorem~(DCT) and \ref{ass:norm-convergence-resolvent-path} imply
    \begin{multline*}
        \lim_{n\rightarrow\infty} \left\|(-\Abase)^{-\alpha} - \En (-\Anbase)^{-\alpha} \Pn \right\|_{\linopo{X}} \\ 
        \overset{\text{DCT}}{\le} \tfrac{1}{2\pi \im}  \int_{\hat{\Gamma}}\lambda^{-\alpha}  \lim_{n\rightarrow\infty} \left\|R(\lambda, -\Abase) - \En R(\lambda, -\Anbase) \Pn  \right\|_{X}  \dlambda \overset{\mathrm{\ref{ass::convergence-semigroups}}}{=} 0.
    \end{multline*}
    Hence, $(-\Anbase)^{-\alpha}$ converges in the norm topology to $(-\Abase)^{-\alpha}$.
\end{proof}

\begin{lemma}
    \label{lem:EnBn_in_Xmin}
    For a setup as described in \Cref{problem:cs} with $\alpha \in (0,1)$ such that \ref{ass:convergence-control-operators} is fulfilled, it holds that
    \begin{equation*}
        \|\Bshort \|_{\linop{\tilde{U}}{\Xminalpha}} \le \mue \lim_{n\rightarrow \infty} \| (-\Anbase)^{-1+\alpha} \Bnshort \|_{\linop{\tilde{U}}{\Xn}}.
    \end{equation*}
\end{lemma}

\begin{proof}
    When considering
    \begin{align*}
        &\quad\; \|((-\Abase)^{-1+\alpha} \Bshort - \En (-\Anbase)^{-1+\alpha} \Bnshort) u\|_X \\
        &\le \| ((-\Abase)^{-1+\alpha} \Bshort - (-\Abase)^{-1+\alpha} \En \Bnshort) u \|_X + \| ((-\Abase)^{-1+\alpha} \En \Bnshort -\En (-\Anbase)^{-1+\alpha} \Bnshort) u\|_X \\
        &= \| (\Bshort - \En\Bnshort)u \|_{\Xminalpha} + \| (\En (-\Anbase)^{-1+\alpha} \Pn (-\Abase)^{1-\alpha} - \Id)(-\Abase)^{-1+\alpha}\En\Bnshort u \|_X \\
        &\le\| (\Bshort - \En\Bnshort)u \|_{\Xminalpha} + \| \En (-\Anbase)^{-1+\alpha} \Pn (-\Abase)^{1-\alpha} - \Id\|_{\linopo{X}} \underbrace{\|(-\Abase)^{-1+\alpha}\En\Bnshort u \|_X }_{\le \text{const.}},
    \end{align*}
    the first term vanishes with $n\rightarrow \infty$ according to \ref{ass:convergence-control-operators} and the right most factor is bounded due to the same reason. Finally, the factor 
    \begin{align*}
    	\| \En (-\Anbase)^{-1+\alpha} \Pn (-\Abase)^{1-\alpha} - \Id\|_{\linopo{X}}
    \end{align*}
    vanishes with $n\rightarrow\infty$ due to \ref{ass:norm-convergence-resolvent-path} and \Cref{lem:frac_power_neg}. Hence, we conclude
    \begin{align*} 
    	\En (-\Anbase)^{-1+\alpha} \Bnshort \narrowinftySOT{U}{X} (-\Abase)^{-1+\alpha} \Bshort,
    \end{align*}
    which completes the proof.
\end{proof}

\begin{lemma}
    \label{lem:flow_no_x_0_estimation_2}
    For a setup as described in \Cref{problem:cs} with
    \begin{equation*}
        \| (-\Abase)^\alpha \S(t)  \|_{\linopo{X}} \le  K_1 \mexp^{- \omega t} + K_2 t^{-\alpha} \mexp^{- \omega t}
    \end{equation*}
    with $\alpha \in (0,1)$, define
    \begin{equation*}
        \kappa(t) \coloneqq \frac{K_1}{\omega}  + K_2 \omega^{-\alpha} \Gamma(\alpha).
    \end{equation*}
    Then, the mild solution to \eqref{eq:cs_base} with initial value $\state_0 = 0$ can be limited for $u \in \ZUt$ by 
    \begin{equation*}
        \| \flowx{t}{0}{u}  \|_X
        \le  \kappa(t)\left( \| \Bshort - \En \Bnshort\|_{\linop{U}{X_{-1+\alpha}}} + \| \En \Bnshort \|_{\linop{U}{X_{-1+\alpha}}}\right) \|u\|_{\ZUt}
    \end{equation*}
    for every $n\in\N$.
\end{lemma}
\begin{proof}
    Using $\Bshort = (\Bshort + \En\Bnshort)-\En\Bnshort$ and $\Linfty$-admissibility of the associated systems, we obtain
    \begin{align*}
        \|\flowx{t}{0}{u} \|_X &= \left\| \int_0^t {\Smin (t-s)\Bshort u(s)} \ds \right\| _{X}\\ & \le \left\| \int_0^t \Smin (t-s)(\Bshort - \En \Bnshort)u(s)\ds \right\|_X + \left\| \int_0^t \Smin (t-s)\En \Bnshort u(s)\ds \right\| _{X}. 
    \end{align*}
    Leveraging the analyticity of the semigroup and fractional power spaces, we bound the first term as
    \begin{align*}
        & \quad \left\|  \int_0^t \Smin(t-s)\left( \Bshort - \En \Bnshort \right)u(s) \ds \right\|_X \\ 
        &\le  \int_0^t \left\| \Smin(t-s)(-\Abase)^{1-\alpha} (-\Abase)^{-1 + \alpha}\left( \Bshort - \En \Bnshort \right)u(s)\right\|_{X}  \ds \\
        & \le \int_0^t \left\| \Smin(t-s)(-\Abase)^{1-\alpha}  \right\|_{\linop{(D((-\Abase)^{1-\alpha}), \|\cdot \|_X)}{X}} \| (-\Abase)^{-1 + \alpha}\left( \Bshort - \En \Bnshort \right) u(s)\| _X \ds\\
        & \le \int_0^t \left(K_1 \mexp^{- \omega (t-s)} + K_2 (t-s)^{-\alpha} \mexp^{- \omega (t-s)}\right) \| \Bshort - \En \Bnshort\|_{\linop{U}{\Xminalpha}} \|u(s) \|_{U} \ds\\
        & \le \|u\|_{\ZUt}\|\Bshort - \En \Bnshort \|_{\linop{U}{X_{-1+\alpha}}} \left[ \int_0^t K_1 \mexp^{- \omega (t-s)} \ds + \int_0^t K_2 (t-s)^{\alpha - 1} \mexp^{- \omega (t-s)} \ds \right]\\ 
        & \overset{r \coloneqq t-s}{=} \|u\|_{\ZUt}\|\Bshort - \En \Bnshort \|_{\linop{U}{X_{-1+\alpha}}}  \left[ -K_1 \int_{t}^0 \mexp^{- \omega r} \dr - K_2 \int_{t}^0 r^{\alpha-1} \mexp^{-\omega r} \dr \right]\\ 
        & \overset{z = \omega r}{=}  \|u\|_{\ZUt} \|\Bshort - \En \Bnshort \|_{\linop{U}{X_{-1+\alpha}}} \left[\frac{K_1}{\omega}\mexp^{-\omega r} \bigg|_{t}^0  - K_2 \omega^{-\alpha} \int_{t}^0 z^{\alpha-1} \mexp^{- z} \dz \right]\\
        & \le \|u\|_{\ZUt} \|\Bshort - \En \Bnshort \|_{\linop{U}{X_{-1+\alpha}}} \left[\frac{K_1}{\omega}  + K_2 \omega^{-\alpha} \Gamma(\alpha)
        \right].
    \end{align*}
    Similarily one yields for the second term 
    \begin{equation*}
        \left\|  \int_0^t \Smin(t-s)\En \Bnshort u(s) \ds \right\|_X \le \|u(s)\|_{\ZUt} \|\En \Bnshort \|_{\linop{U}{X_{-1+\alpha}}} \left[\frac{K_1}{\omega}  + K_2 \omega^{-\alpha} \Gamma(\alpha)\right],
    \end{equation*}
    which shows the desired inequality.
\end{proof}

Summarizing the previous results, we obtain the following solution to \Cref{problem:cs}.

\begin{theorem}
    \label{thm:dist_contr}
    Considering the setting described in \Cref{problem:cs} and assume further that $\{\Sn(t)\}_{t\ge 0}$ is of type $(M_n, -\omega_n)$ for $n \in\N$ and that $(M_n)_{n\in \N}$ and $(\omega_n)_{n\in \N}$ are Cauchy sequences. For fixed $\alpha\in(0,1)$ such that \ref{ass:convergence-control-operators} holds and $\theta \in \left(\tfrac{\pi}{2}, \tfrac{\pi}{2} + \sectorbound\right)$ define
    \begin{align*}
        \hat{M} & \coloneqq \mue \mup \lim_{n\rightarrow \infty} M_n , \quad &
        \hat{\omega} & \coloneqq \lim_{n \rightarrow \infty} \omega_n, \\
        \hat{D} &= \mue\mup\sup_{n\in  \N} \sup_{\lambda \in \C \setminus \overline{\othersector_\sectorbound}} \|\res{\lambda}{-\Anbase}\|_{\Xn}(|\lambda| +1), \quad &
        K_1 &\coloneqq \frac{ \hat{\omega} \hat{M} }{\Gamma(1-\alpha)} \int_0^\infty s^{-\alpha} \mexp^{-\hat{\omega} s}\ds,\\
        K_2 &\coloneqq \frac{\hat{D}}{\Gamma(1-\alpha)\pi \cos(\theta)} \int_0^\infty s^{-\alpha}  (1+s)^{-1} \ds, \quad &
        \kappa &\coloneqq \frac{K_1}{\omega}  + K_2 \hat{\omega}^{-\alpha} \Gamma(\alpha).
    \end{align*}
    Then, the control system $\Sigma(\Abase, \Bshort; X, U)$ is $\Linfty$-\ISS with \ISS gains
    \begin{equation}
    	\label{eqn:conv_char_fun}
        \hat{\beta}(s, t) \vcentcolon= \hat{M} \,\mexp^{-\hat{\omega} t} \,s, \quad 
        \hat{\gamma}(s) \vcentcolon= \mue \kappa \lim_{n\rightarrow\infty}\|(-\Anbase)^{-1+\alpha}\Bnshort\|_{\linop{\Xn}{U}}s.
    \end{equation}
\end{theorem}

\begin{proof}
We know by \Cref{lem:omega_conv,lem:growth_bound_frac_power}, that 
\begin{equation} \label{eq:growthbounds_approx}
    \|\S(t)\|_{\linopo{X}} \le \hat{M} \mexp^{- \hat{\omega} t} ,\text{ and }\quad
    \| (-\Abase)^{\alpha} \S(t) \|_{\linopo{X}} \le K_1 \mexp^{- \hat{\omega} t}  + K_2 t^{-\alpha} \mexp^{- \hat{\omega} t}.
\end{equation}
Utilizing the linearity of the problem, we compute 
\begin{equation*} \label{eq:proof_split_x0u}
    \| \flowx{t}{\stateinit}{u} \|_X \le \|  \flowx{t}{\stateinit}{0}\|_X + \| \flowx{t}{0}{u}\|_X .
\end{equation*}
Applying \Cref{lem:flow_no_x_0_estimation_2} and equation~\eqref{eq:growthbounds_approx}, we obtain
\begin{align*} 
    \| \flowx{t}{\stateinit}{u} \|_X &\le \hat{M} \mexp^{-\hat{\omega} t} \|\stateinit \|_X +  \kappa  \lim_{n\rightarrow \infty}  (\|\Bshort - \En \Bnshort \|+ \|\En\Bnshort \|)  \|u(s)\|_{\ZUt} \\
    &= \hat{M} \mexp^{-\hat{\omega} t} \|\stateinit \|_X + \kappa\lim_{n\rightarrow \infty} \|\En\Bnshort \| \|u(s)\|_{\ZUt}
\end{align*}
where we droped the norm subscripts ${\linop{U}{X_{-1+\alpha}}}$ in the terms involving the control operator for readability. The second line follows from \ref{ass:convergence-control-operators} and \Cref{lem:EnBn_in_Xmin}.
\end{proof}

\section{Approximation of \ISS gains for boundary control systems}
\label{sec:boundary-control-results}

The previous section, specifically \Cref{thm:dist_contr}, provides a constructive answer to \Cref{problem:cs}. In particular, it shows that the \ISS gains of an infinite-dimensional control system with a weakly unbounded control operator can be computed by analyzing finite-dimensional approximations of \eqref{eq:cs_base} and then passing to the limit. Building on these results, we now turn to \Cref{problem:bcs}, which extends \Cref{problem:cs} by establishing the convergence of the control operators
\begin{equation}
    \label{eq:boundary-aim}\tag*{\ref{ass:convergence-control-operators}}
    \lim_{n \rightarrow \infty} \| \Bshort - \En \Bnshort\|_{\linop{U}{\Xminalpha}} = 0,
\end{equation}
from properties of the boundary operator such as \ref{ass:bcs:bop-right-invertible}. To bridge the view from \Cref{sec:distributed-control-results} to the boundary operator-focused view, recall that the (distributed) control operators are constructed using the Fattorini trick (\Cref{lem:fattorini-trick}), yielding
\begin{equation*}
    \Bshort = \Ainit \boprinv - \Amin \boprinv.
\end{equation*}
We will apply the same trick in the finite-dimensional setting. We use the one-dimensional heat equation in \Cref{example:boundary-discretization} to illustrate how the Fattorini trick applies in a concrete finite-difference setting. With this example in mind, we then set up different systems (see the forthcoming \Cref{fig:bcsystems}) that we use to establish the link between the boundary control system and its approximation via finite-dimensional control systems in \Cref{subsec:convergenceBounderyOp}, where we then also present the main result for this section, namely the convergence of the boundary control operators~\ref{eq:boundary-aim}.

\subsection{Motivational example}
\label{example:boundary-discretization}
Consider the one-dimensional heat equation on the spatial domain $\Omega = (0,1)$ with diffusion constant $\diffconst$ and Dirichlet boundary control, i.e.,
    \begin{equation}
    	\label{eqn:heat}
        \left\{\quad
        \begin{aligned}
            \statedt(t, \spaceVar) & = \diffconst \,\partial_\spaceVar^2 \state(t, \spaceVar) \quad&\mathrm{on}\; \timeInt \times (0,1),\\
            \state(t, 0) & = \inp_0(t) \quad&\mathrm{on}\; \timeInt,\\
            \state(t, 1) & = \inp_1(t) \quad&\mathrm{on}\; \timeInt, \\
            \state(0, \spaceVar) & = \stateinit \quad&\mathrm{on}\; (0,1),\\
        \end{aligned}\right.
    \end{equation}
    with inputs $\inp_0,\inp_1\colon\timeInt\to\R$.
    The corresponding control operator in the sense of~\eqref{eq:bcs_base} is given by the trace operator. In the forthcoming \Cref{fig:bcsystems}, the heat equation~\eqref{eqn:heat} takes the role of the boundary control system \sysA. For this problem, the structure of a classical finite-difference discretization (see \Cref{fig:motivation}), where the interval $[0, 1]$ is split into $n$ equally sized intervals of length $\tfrac{1}{n}$, results in the finite-dimensional control system
    \begin{equation}
        \label{eq:AnBn_DirichletHeat}
        \Anbase = \diffconst \,n^2 \begin{bmatrix}
            -2 & 1 & 0 & \cdots & 0 \\
            1 & -2 & 1 & \cdots & 0 \\
            0 & 1 & -2 & \cdots & 0 \\
            \vdots & \vdots & \vdots & \ddots & \vdots \\
            0 & 0 & 0 & \cdots & -2
        \end{bmatrix} \in \R^{(n-1) \times (n-1)}, \quad 
        \Bnshort = \diffconst n^2 \begin{bmatrix}
            1 & 0\\
            0 & 0\\
            \vdots & \vdots\\
            0 & 1
        \end{bmatrix} \in \R^{(n-1) \times 2}.
    \end{equation}
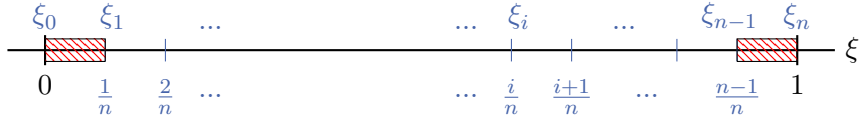
\begin{figure}[ht]
        \centering
        \begin{tikzpicture}
            \draw[-,thick] (-0.5,0)--(10.5,0) node[right]{$\coord$};
            \draw[-, thick] (0,-0.2)--(0,0.2) node[below, yshift=-0.4cm]{0};
            \draw[-, thick] (10,-0.2)--(10,0.2) node[below, yshift=-0.4cm]{1};

            \draw[-, kit-blue] (0.8,-0.15)--(0.8,0.15) node[below, yshift=-0.4cm]{$\tfrac{1}{n}$};
            \draw[-, kit-blue] (1.6,-0.15)--(1.6,0.15) node[below, yshift=-0.4cm]{$\tfrac{2}{n}$};
            \draw[-, kit-blue] (6.2,-0.15)--(6.2,0.15) node[below, yshift=-0.4cm]{$\tfrac{i}{n}$};
            \draw[-, kit-blue] (7,-0.15)--(7,0.15) node[below, yshift=-0.4cm]{$\tfrac{i+1}{n}$};
            \draw[-, kit-blue] (9.2,-0.15)--(9.2,0.15) node[below, yshift=-0.4cm]{$\tfrac{n-1}{n}$};
            \draw[-, kit-blue] (8.4,-0.15)--(8.4,0.15) ;

            \node[kit-blue] (xi) at   (6.3,  0.45){$\coord_i$} ;
            \node[kit-blue] (xone) at (0.9,  0.45){$\coord_1$} ;
            \node[kit-blue] (xone) at (0.0,  0.45){$\coord_0$} ;
            \node[kit-blue] (xn) at   (9.1,  0.45){$\coord_{n-1}$} ;
            \node[kit-blue] (xn) at   (10.0, 0.45){$\coord_{n}$} ;

            \node[kit-blue] at (2.2, -0.6){$...$};
            \node[kit-blue] at (5.6, -0.6){$...$};
            \node[kit-blue] at (8, -0.6){$...$};
            \node[kit-blue] at (7.7, 0.3){$...$};
            \node[kit-blue] at (5.6, 0.3){$...$};
            \node[kit-blue] at (2.2, 0.3){$...$};

            \draw [pattern=north west lines, pattern color=red] (0.8,-0.15)--(0,-0.15) -- (0, 0.15) -- (0.8,0.15) --(0.8,-0.15) ;
            \draw [pattern=north west lines, pattern color=red] (9.2,-0.15)--(10,-0.15) -- (10, 0.15) -- (9.2,0.15) --(9.2,-0.15) ;
        \end{tikzpicture}
        \caption{Modelling of finite difference discretization. The red areas depict the intervals fully determined by the boundary conditions.}
        \label{fig:motivation}
\end{figure}
    Note that in this representation the boundary nodes $x_0$ and $x_n$ do not appear explicitly, as illustrated in \Cref{fig:motivation}. In the view of \Cref{fig:bcsystems}, the discretized system~\eqref{eq:AnBn_DirichletHeat} is of type~\sysD, and we immediately observe that \sysD is a control system (cf.~\Cref{def:control_system}) while \sysA constitutes a boundary control system (cf.~\Cref{def:boundary_control_ssystem}). To bridge this gap, we introduce an intermediate finite-dimensional boundary control system (system \sysC in \Cref{fig:bcsystems}) that can then be transformed to the control system \sysD using a Fattorini-type transformation. 

For our example of the heat equation, the intermediate system \sysC is given by the matrices 
    \begin{align} 
        \label{eq:SysC_DirichletHeat_Ainit}
        \Aninit &= \diffconst n^2 \begin{bmatrix}
            1 & -2 & 1 & 0 & \cdots & 0 \\
            0 & 1 & -2 & 1 & \cdots & 0 \\
            \vdots & \vdots & \vdots & \ddots & &\vdots \\
            0 & 0 & 0 & \cdots& -2 & 1\\
        \end{bmatrix}\in \R^{(n-1) \times (n+1)}, \\ 
        \label{eq:SysC_DirichletHeat_bopn}
        \bopn &= \begin{bmatrix}
            1 & 0 & \dots & 0 & 0 \\
            0 & 0 & \dots & 0 & 1
        \end{bmatrix} \in \R^{2 \times (n+1) }, \\ 
        \label{eq:SysC_DirichletHeat_descMatrix}
        \descMatrix &= \begin{bmatrix}
            0 & 1 & 0 & \cdots & 0 & 0\\
            0 & 0 & 1 & \cdots & 0 & 0\\
            \vdots & \vdots &  & \ddots & \vdots & \\
            0 & 0 & 0 & \cdots & 1 & 0\\
        \end{bmatrix} \in \R^{(n-1) \times (n+1)},
    \end{align}
    which is consistent in the sense that $\Anbase = \Aninit |_{\ker \bopn}$.
    We can interpret $\Aninit$ as the same finite-difference discretization as $\Anbase$, except that the boundary information has not yet been inserted and $\bopn$ as a discretization of the boundary operator $\bop$. 
    
For the application of the Fattorini trick first observe $\ker \bopn = \R^{n+1}\setminus \mathrm{span} \left\{ \unitvector{1}, \unitvector{n+1}\right\}$, and hence
    \begin{gather*}
        \Anbase = \Aninit |_{\ker \bopn}, \quad
        \bopnrinv = \begin{bmatrix}
            1 & 0 \\
            0 & \vdots \\
            \vdots  & 0 \\
            0 & 1
        \end{bmatrix} \in \R^{(n+1) \times 2},\\
        \Bnshort = \Aninit \bopnrinv - \Anbase \Rn \bopnrinv = \diffconst n^2 \begin{bmatrix}
            1 & 0\\
            0 & \vdots\\
            \vdots & 0\\
            0 & 1
        \end{bmatrix} \in \R^{(n-1) \times 2},
    \end{gather*} 
    where $\Rn$ is the canonical projection from $\R^{n+1}$ to $\R^{n-1}$.
    
\subsection{Convergence of the boundary operator}
\label{subsec:convergenceBounderyOp}

Building on the previous example, we now formalize the transformations, which are summarized in \Cref{fig:bcsystems}. Systems \sysA and~\sysB depict the infinite-dimensional system before and after applying the Fattorini trick, respectively. For each, we consider corresponding approximate systems, namely \sysC and \sysD, respectively. Since these discretizations are tightly connected, several consistency assumptions need to be made, such as $\tildeEn=\En \Rn$, which are stated to full extent in \Cref{thm:bound_cont_approx}. 

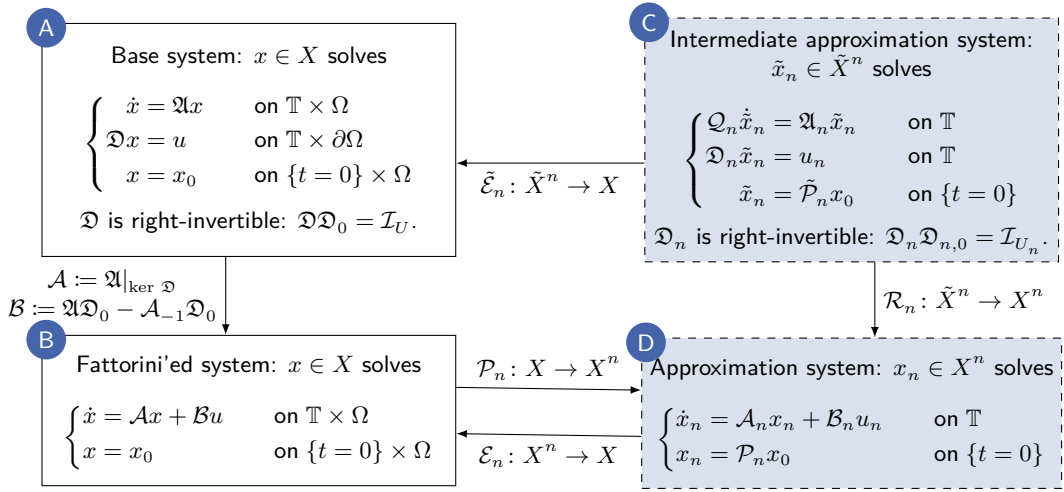
\begin{figure}[t]
    \centering
    \begin{tikzpicture}[ 
        node distance = 1cm, 
        node font = {\footnotesize\sffamily},
        >=latex,block/.style={
            draw, 
            fill=white, 
            rectangle, 
            minimum height=3em, 
            minimum width=8em},
        ]
        
        \node[block, minimum width=5.5cm, minimum height=3.1cm] (baseeq) {\makecell{Base system: $\state \in X$ solves \\[1em] $
        \left\lbrace
            \begin{aligned}
                \statedt &= \Ainit \state \quad &\text{on}\;& \timeInt \times \Omega \\
                \bop \state &= \inp \quad& \text{on}\; &\timeInt \times \partial\Omega  \\
                \state &= \stateinit \quad& \text{on}\; &\{t=0\} \times \Omega
            \end{aligned}\right.
        $\\[2em] $\bop$ is right-invertible: $\bop\boprinv = \mathcal{I}_U$.}};

        \node[block, dashed, node distance=2.5cm, right =of baseeq,  minimum width=5.5cm, minimum height=3.2cm, fill=kit-blue!20] (intuitive) {\makecell{Intermediate approximation system:\\ $\stateCn \in \tildeXn$ solves \\[1em] $
        \left\lbrace
            \begin{aligned}
                \descMatrix \stateCndt &= {\Aninit} \stateCn \quad &\text{on}\;& \timeInt  \\
                \bopn \stateCn&=\inpn \quad& \text{on}\; &\timeInt \\
                \stateCn&= \tildePn \stateinit \quad& \text{on}\; &\{t=0\} 
            \end{aligned}\right.
        $\\[2em] $\bopn$ is right-invertible: $\bopn\bopnrinv = \mathcal{I}_{U_n}$.}};

        \node[block, dashed, below =of intuitive, minimum width=5.5cm, minimum height=2cm, fill=kit-blue!20] (final) {\makecell{Approximation system: $\stateDn \in \Xn$ solves \\[1em] $\left\lbrace
            \begin{aligned}
                \stateDndt &= {\Anbase} \stateDn + \Bnshort \inpn \quad &\text{on}\;& \timeInt  \\
                \stateDn&= \Pn \stateinit \quad& \text{on}\; &\{t=0\} 
            \end{aligned}\right.
        $}};

        \node[block, below of=baseeq, node distance=3.6cm, minimum width=5.5cm, minimum height=2cm] (disteq) {\makecell{Fattorini'ed system: $\state \in X$ solves \\[1em] $    \left\lbrace
            \begin{aligned}
                \statedt &= \mathcal{A} \state + \Bshort \inp \quad &\text{on}\;& \timeInt \times \Omega \\
                \state &= \stateinit \quad& \text{on}\; &\{t=0\} \times \Omega
            \end{aligned}\right.
        $}};

        \begin{scope}[transform canvas={xshift=2.5em, yshift=-2.5em}]
            \node[above left= of disteq, circle,fill=kit-blue,inner sep=0pt,minimum size=15pt, text=white] (A) {B};
            \node[above left= of baseeq, circle,fill=kit-blue,inner sep=0pt,minimum size=15pt, text=white] (B) {A};
            \node[above left= of intuitive, circle,fill=kit-blue,inner sep=0pt,minimum size=15pt, text=white] (C) {C};           
            \node[above left= of final, circle,fill=kit-blue,inner sep=0pt,minimum size=15pt, text=white] (D) {D};
        \end{scope}

        \begin{scope}[transform canvas={xshift=-.8em}]
            \draw[latex-] (disteq.north) to node [left] {\makecell{$\mathcal{A} \coloneqq \Ainit|_{\ker\, \mathfrak{D}}$\\$\Bshort \coloneqq \Ainit \boprinv - \mathcal{A}_{-1} \boprinv$}} (baseeq.south);

        \end{scope}
        \begin{scope}[transform canvas={xshift=.8em}]

            \draw[-latex, xshift = 4pt] (intuitive.south) to node [auto] {${\Rn}\colon  \tildeXn \rightarrow \Xn$} (final.north);
        \end{scope}

        \begin{scope}[transform canvas={xshift=0cm}]
            \draw[latex-] ([yshift=-.8em]disteq.east) to node [below] { $\En \colon \Xn \rightarrow X$}([yshift=-.8em]final.west);
            \draw[-latex] ([yshift=.8em]disteq.east) to node [above] { $\Pn \colon X \rightarrow \Xn$} ([yshift=.8em]final.west);
            
            \draw[latex-] ([yshift=-.8em]baseeq.east) to node [below] {$\tildeEn\colon \tildeXn \rightarrow X$} ([yshift=-.8em]intuitive.west);
        \end{scope} 



    \end{tikzpicture}    
    \caption{Systems under consideration.} \label{fig:bcsystems}
\end{figure}

In practice, one often omits \sysC and works directly with the discretization \sysD. Nevertheless, we want to emphasize again that \sysC is essential for the following results: it preserves the structure of \sysA while allowing us to track precisely how the Fattorini transformation carries over to the discrete level. We call \sysC the pre-boundary closure discretization.

\begin{theorem}\label{thm:bound_cont_approx}
    Given a setting as described in \Cref{problem:bcs} and the transformations as introduced in \Cref{fig:bcsystems} fulfilling 
    \begin{enumerate}
            \item $\tildeEn = \En\Rn$,
            \item $\range\left(\descMatrix\right) \subseteq \Xn$ and $\range\left(\Aninit\right) \subseteq \Xn$, and
            \item $\Anbase = \Aninit |_{\ker \bopn}$ generates a strongly continuous semigroup.
    \end{enumerate}
    Further, assume that
    \begin{enumerate}[label=(B\arabic*)]
        \addtocounter{enumi}{4}
        \item the right inverses of the boundary operators converge, i.e., $\tildeEn \bopnrinv\narrowinftySOT{U}{\Xalpha} \boprinv$, \label{thm:ass:conv-bopnrinv}
        \item there is a $C$ independent of $n$ such that $\|\Aninit \bopnrinv \|_{\linop{U}{\Xn}} \le C$ for all $n\in \N$, and \label{thm:ass:bopnrinv-uniformly-bounded}
        \item \label{ass:conv_bcs_1} $\En \Anmin^{-1} \Aninit \bopnrinv \narrowinftySOT{U }{X} \Abase^{-1} \Ainit \boprinv$.
    \end{enumerate}
    Then, 
    \begin{equation}
        \label{eq:approximation_bn_in_Xminalpha}
        \En \Bnshort \narrowinftySOT{U}{\Xminalpha} \Bshort .
    \end{equation}
\end{theorem}

\begin{figure}[b]
    \centering
    \begin{tikzpicture}[ 
        node distance = 1cm, 
        node font = {\footnotesize\sffamily},
        >=latex,block/.style={
            draw, 
            fill=white, 
            rectangle, 
            minimum height=3cm, 
            minimum width=5.5cm}]

        \node[block, dashed, fill=kit-blue!20] (intuitive) {\makecell{$\stateCn \in \tildeXn$ solves \\[1em] $
            \begin{aligned}
                \descMatrix \stateCndt &= \Aninit \stateCn \quad &\text{on}\;& \timeInt  \\
                \bopn \stateCn&=\inp \quad& \text{on}\; &\timeInt \\
                \stateCn&= \tildePn \stateinit \quad& \text{on}\; &\{t=0\} 
            \end{aligned}
        $\\[2em] $\bopn$ is right-invertible: $\bopn\bopnrinv = \mathcal{I}_{U}$.}};

        \node[block, dashed, right=2cm of intuitive, minimum width=8.5cm] (zntilde) {\makecell{$\stateEn \in \tildeXn$ solves \\[1em] $
            \begin{aligned}
                \descMatrix \stateEndt &= \Aninit \stateEn + \Aninit \bopnrinv \inp - \descMatrix \bopnrinv \inpdt\quad &\text{on}\;& \timeInt  \\
                \bopn \stateEn&=0 \quad& \text{on}\; &\timeInt \\
                \stateEn&= \tildePn \stateinit - \bopnrinv\inp \quad& \text{on}\; &\{t=0\} 
            \end{aligned} 
        $}};

        \node[block, dashed, below =of intuitive, fill=kit-blue!20, minimum height=2cm] (final) {\makecell{$\stateDn \in \Xn$ solves \\[1em] $
            \begin{aligned}
                \stateDndt &= {\Anbase} \stateDn + \Bnshort \inp \quad &\text{on}\;& \timeInt  \\
                \stateDn&= \Pn \stateinit \quad& \text{on}\; &\{t=0\} 
            \end{aligned}
        $}};

        \node[block, dashed, below =1cm of zntilde, minimum width=8.5cm, minimum height=2cm] (almostfinal) {\makecell{ $\stateFn \in \Xn$ solves \\[1em] $
            \begin{aligned}
                \stateFndt &= {\Anbase} \stateFn + \Aninit \bopnrinv \inp - \descMatrix \bopnrinv \inpdt&\text{on}\;& \timeInt  \\
                \stateFn&= \Pn \stateinit - \Rn\bopnrinv\inp\;& \text{on}\; &\{t=0\} 
            \end{aligned}
        $}};
        \begin{scope}[transform canvas={xshift=2.5em, yshift=-2.5em}]
            \node[above left= of intuitive, circle,fill=kit-blue,inner sep=0pt,minimum size=15pt, text=white] (C) {C};
            \node[above left= of zntilde, circle,fill=kit-blue,inner sep=0pt,minimum size=15pt, text=white] (E) {E};
            \node[above left= of almostfinal, circle,fill=kit-blue,inner sep=0pt,minimum size=15pt, text=white] (F) {F};
            \node[above left= of final, circle,fill=kit-blue,inner sep=0pt,minimum size=15pt, text=white] (D) {D};
        \end{scope}

        \begin{scope}[transform canvas={xshift=-.8em}]
            \draw[latex-] (intuitive.south) to node [left] {${\Fn}\colon {\Xn} \rightarrow \tildeXn$} (final.north);
        \end{scope}
        \begin{scope}[transform canvas={xshift=.8em}]
            \draw[-latex, xshift = 4pt] (zntilde.south) to node [auto] {${\Rn}\colon  \tildeXn \rightarrow \Xn$} (almostfinal.north);
            \draw[-latex, xshift = 4pt] (intuitive.south) to node [auto] {${\Rn}\colon  \tildeXn \rightarrow \Xn$} (final.north);
        \end{scope}


        \draw[-latex] (intuitive.east) to node [right, xshift=-0.9cm, yshift=0.5cm]{\makecell{$\begin{aligned}&\stateEn \coloneqq \stateCn - \\& \bopnrinv \inp\end{aligned}$}} (zntilde.west);
        \draw[latex-] (final.east) to node [right, xshift=-0.9cm, yshift=0.5cm]{\makecell{$\begin{aligned}&\stateDn = \stateFn +\\ & \Rn\bopnrinv \inp\end{aligned}$}} (almostfinal.west);

    \end{tikzpicture}    
    \caption{Fattorini trick applied to systems \sysC and \sysD from \Cref{fig:bcsystems}.} \label{fig:bcsystems_2}
\end{figure}
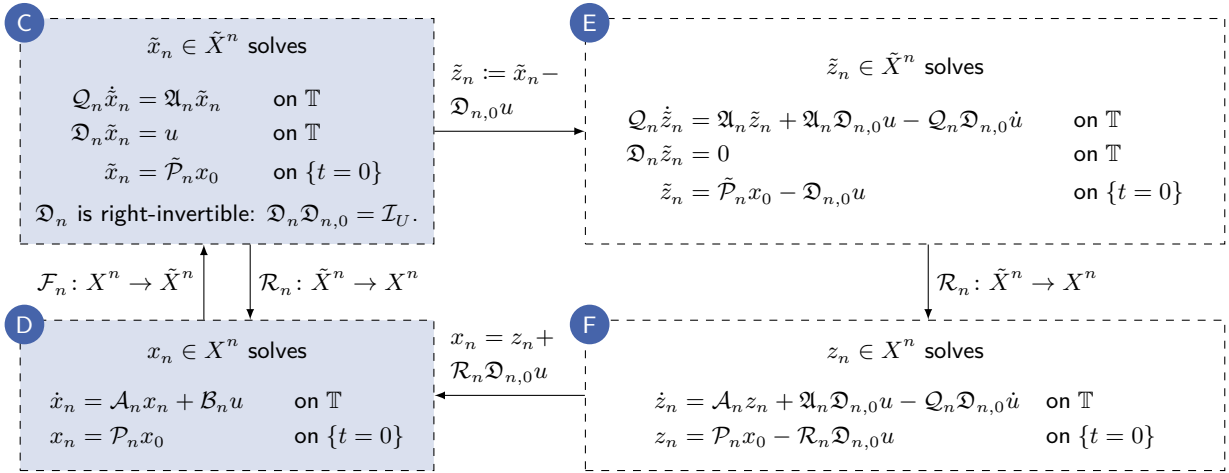

\begin{proof}
The proof is structured in three steps. First, we investigate how to apply the Fattorini trick to the intermediate approximation system and derive the resulting boundary operator $\Bnshort$. Second, knowing $\Bshort$ and $\Bnshort$ allows us to link the systems \sysA and \sysD. Third, we show the desired convergence~\eqref{eq:approximation_bn_in_Xminalpha}.

\textbf{Step 1: Fattorini on intermediate approximation system}\\
Invoking the Fattorini trick at the discrete level, the resulting family of systems is depicted in \Cref{fig:bcsystems_2}. The mild solution of \sysF reads
\begin{multline}
    \label{eq:fin_dim_fattorini}        
    \stateFn(t)= \Sn(t)(\Pn \stateinit - \Rn \bopnrinv \inp(0)) + \int_0^t  \Sn (t-s) \Aninit \bopnrinv \inp(s)\ds \\
    - \int_0^t  \Sn (t-s) \descMatrix\bopnrinv \inpdt(s)\ds.
\end{multline}
Integration by parts of the last summand yields
\begin{align*}
    \int_0^t  \Sn (t-s) &\descMatrix \bopnrinv \inpdt(s)\ds \\ 
    &= \left[\Sn (t-s) \descMatrix \bopnrinv \inp(s) \right]_0^t + \int_0^t  \Snmin (t-s) \Anmin\descMatrix \bopnrinv \inp(s)\ds \\
    &= \descMatrix \bopnrinv \inp(t)  - \Sn (t) \descMatrix \bopnrinv \inp(0)  + \int_0^t  \Snmin (t-s) \Anmin\descMatrix \bopnrinv \inp(s)\ds. 
\end{align*}
Substituting this expression into \eqref{eq:fin_dim_fattorini} yields
\begin{align*}
    \stateFn(t) &= \Sn(t)\Pn \stateinit + \int_0^t  \Sn (t-s) \Aninit \bopnrinv \inp(s)\ds - \Rn \bopnrinv \inp(t)  \\ & \phantom{=}- \int_0^t  \Snmin (t-s) \Anmin\Rn \bopnrinv \inp(s)\ds \\
    &= \Sn(t)\Pn \stateinit  + \int_0^t  \Snmin (t-s) \left[\Aninit -\Anmin\Rn \right] \bopnrinv \inp(s)\ds - \Rn \bopnrinv \inp(t),
\end{align*}
which we then revert to $\stateDn$ to obtain
\begin{align*}
    \stateDn(t)= \Sn(t) \Pn \stateinit  + \int_0^t  \Snmin(t-s)\left[\Aninit \bopnrinv -\Anmin\Rn \bopnrinv \right] \inpn(s)\ds.
\end{align*}
From this, we define the corresponding control operator $\Bnshort$ for the distributed control in system \sysD as
\begin{equation*}
    \Bnshort = \Aninit \bopnrinv -\Anmin\Rn \bopnrinv.
\end{equation*}
Since each space $\Xn$ is finite dimensional, the associated extrapolation space coincides with $\Xn$ itself (up to the choice of norm) and similarly for the involved operators. Hence,
\begin{equation*}
    \Bnshort = \Aninit \bopnrinv -\Anbase \Rn \bopnrinv.
\end{equation*}

\textbf{Step 2: Linking \sysA and \sysD}
The formulation of $\Bnshort$ allows us to systematically link the controls of systems \sysA and \sysD via 
\begin{align}
    (\Bshort - \En \Bnshort)u &= (\Ainit\boprinv  - \Amin  \boprinv)u - \En (\Aninit \bopnrinv - \Anbase \Rn \bopnrinv) u \\
    &= (\Ainit\boprinv - \En \Aninit  \bopnrinv)u  \label{eq:conv_bcs_1}\\
    &\phantom{=}\quad + (\En \Anbase \Rn \bopnrinv - \Amin \boprinv ) u.\label{eq:conv_bcs_2}
\end{align}

\textbf{Step 3: Convergence}
We will study the convergence of the terms in \eqref{eq:conv_bcs_1} and \eqref{eq:conv_bcs_2} separately.
To show the convergence of \eqref{eq:conv_bcs_1} in the $\Xminalpha$-norm, we employ \Cref{lem:interpolation-fractional-power}, which implies the existence of $L>0$ such that
\begin{align}
    \| (-\Abase)^{-1+\alpha} x \|_{X} &\le L \| (-\Abase)^{-1} x\|_X^{\frac{0+ (-1+\alpha)}{0+1}}\| x\|_X^{\frac{(-1+\alpha) +1 }{0 +1}} \nonumber\\ 
    &= L \| \Abase^{-1} x\|_X^{1-\alpha}\| x\|_X^{\alpha} = L \|x\|_{\Xmin}^{1-\alpha} \| x \|_{X}^{\alpha} \label{eq:interpolation-inequality-alpha-neg}
\end{align}
Consequently, the RHS in \eqref{eq:conv_bcs_1} converges to zero (in the $\Xminalpha$-norm), if
\begin{enumerate}
    \item $\| (\Ainit \boprinv - \En \Aninit \bopnrinv ) u \|_{\Xmin} \narrowinfty 0$, and \label{item:AninitBopnrinv-converge-Xmin}
    \item $\| (\Ainit \boprinv - \En \Aninit \bopnrinv ) u \|_{X} \leq C$ for all $n \in \N$ and $C$ independent of $n$\label{item:AninitBopnrinv-bounded-X}
\end{enumerate}
are satisfied. The second condition follows directly from assumptions \ref{ass:bcs:bop-right-invertible} and \ref{thm:ass:bopnrinv-uniformly-bounded}, while~\ref{item:AninitBopnrinv-converge-Xmin} follows from
\begin{align*}
    \phantom{=} \lim_{n\rightarrow \infty}\| (\Ainit \boprinv - \En \Aninit \bopnrinv ) u \|_{\Xmin} &\overset{\phantom{\text{\ref{ass:norm-convergence-resolvent-path}}}}{=} \lim_{n\rightarrow \infty}\| (\Abase^{-1} \Ainit \boprinv - \Abase^{-1} \En \Aninit \bopnrinv ) u \|_{X} \\
    &\overset{\text{\ref{ass:norm-convergence-resolvent-path}}}{=} \lim_{n\rightarrow \infty}\| (\Abase^{-1} \Ainit \boprinv - \En \Anbase^{-1} \Aninit \bopnrinv ) u \|_{X} = 0
\end{align*}
due to assumption \ref{ass:conv_bcs_1}.

To establish convergence of \eqref{eq:conv_bcs_2} to zero in the $\Xminalpha$-norm, we need to show boundedness of $\Enmin\Anmin \Pn$ to $\Amin$ as operators in $\linop{\Xalpha}{\Xminalpha}$ uniformly in $n\in \N$ since the strong operator convergence of $\En \Rn \bopnrinv$ to $\boprinv$ in $\linop{U}{\Xalpha}$ is already assumed in \ref{thm:ass:conv-bopnrinv}. Since we assumed that $\Abase \in \BIP(X)$ in \ref{ass:bcs:bip}, we employ the identification of the fractional power spaces with the complex interpolation spaces (cf.~\cite[Thm.~11.6.1]{MarS01}) to then show that there exists $C_1 > 0$ independent of $n$ such that 
\begin{equation}\label{eq:estimate-interpolation-spaces}
    \| \En \Anbase \Pn - \Anmin \|_{\linop{[X, D(\Abase)]_\alpha}{[\Xmin, X]_\alpha}} \le C_1
\end{equation}
for all $n\in \N$. Exactness of the interpolation, cf.~\cite[Sec.~4]{Cal64}, then allows us bound the norm in the interpolated space for an operator $\calF$ mapping from $D(A) + X$ to $X + \Xmin$ by 
\begin{equation*}
    \| \calF \|_{\linop{[X, D(\Abase)]_\alpha}{[\Xmin, X]_\alpha}} \le  \|\calF \|^{1-\alpha}_{\linop{X}{\Xmin}} \| \calF \|^\alpha_{\linop{D(\Abase)}{X}}.
\end{equation*}
Applying this to \eqref{eq:estimate-interpolation-spaces} yields that we need to show
\begin{enumerate} 
    \item that there exists $C_2 > 0$ such that 
    \begin{equation*}
        \| \En \Anbase \Pn - \Abase\|_{\linop{D(\Abase)}{X}} = \|\Id - \En\Anbase\Pn \Abase^{-1}  \|_{\linopo{X}} \le C_2,
    \end{equation*}
    \item and that there exists $C_3 > 0$ such that \label{item:Amin-convergence-Xmin}
    \begin{equation*} 
        \| \En \Anbase \Pn - \Amin \|_{\linop{X}{\Xmin}} =  \| \Abase^{-1} \En \Anbase \Pn - \Id \|_{\linopo{X}} \le C_3.
    \end{equation*} 
\end{enumerate}
Both conditions follow directly from assumption \ref{ass:consistent_approx}, which concludes the proof.
\end{proof}

\begin{remark}
    The construction of a right-inverse for the discrete boundary operator $\bopn$ is not spelled out in either the proof or the statement of the theorem. A convenient way to obtain $\bopnrinv$ is to proceed in the following three steps:
    \begin{enumerate}
        \item Construct a discretized boundary control $\inpn(t) \coloneqq \PnU \inp(t)$ with $\PnU \colon \InpSpace \rightarrow \InpnSpace$ being a suitable restriction operator.
        \item Construct a matrix representation of $\bopn$ and compute the right inverse $\bopnrinvU$ using standard linear algebra.
        \item Compute $\bopnrinv \coloneqq \bopnrinvU \PnU$.
    \end{enumerate}
    
    The proof of the previous theorem can be adapted to use $\inpn \in \InpnSpace$ and a suitably adapted convergence criterion
    \begin{equation*}
        \|(\En \bopnrinvU \PnU - \boprinv) u\|_{\Xalpha} \narrowinfty 0.
    \end{equation*}
\end{remark}

\section{Numerical example: Dirichlet controlled heat equation}
\label{sec:numerics}


We continue the investigation of the heat equation with Dirichlet boundary control~\eqref{eqn:heat} from \Cref{example:boundary-discretization}, following similar discussions as in \cite[Sec.~4.1]{ItoK98} and \cite[Sec.~4.1]{HilU25b}.

We consider the problem on $X = L^2(\Omega)$ with $\Ainit \state = \partial^2_\coord \state$ on $D(\Ainit) = H^2(\Omega)$. The boundary conditions are formally enforced by the Dirichlet trace operator $\bop \colon H^1(\Omega) \rightarrow \R^2$ and thus $\ker \bop = H_0^1(\Omega)$ and  $D(\Abase) = H^2(\Omega) \cap H_0^1(\Omega)$. We use the right-inverse of the boundary trace given by 
\begin{equation*}
    \boprinv \colon \R^2 \mapsto D(\Ainit), \qquad \begin{bmatrix}
        u_0\\ u_1
    \end{bmatrix} \mapsto f(x) \coloneqq u_1\, x + u_0\, (1-x).
\end{equation*}
For the approximation, we divide $[0,1]$ into $n$ equally sized intervals as in \Cref{example:boundary-discretization} with $\Xn = (\R^{n-1}, \|\staten\|_{\Xn}^2 \coloneqq \deltacoord^2 \sum_{k=1}^{n-1}|\statenidx{k}|^2)$, and
\begin{equation*}
    \begin{aligned}
        (\En g) (\coord) &\coloneqq \sum_{k=1}^{n-1} g_k B_{n, k}(\coord), \qquad  &
        (\Pn f)_{k} & \coloneqq f(\coord_k)
        , \hspace{1.5em} 1 \le k \le n-1,
    \end{aligned}
\end{equation*}
with $B_{n,k}(\coord)$ being the first order B-spline defined by 
\begin{equation*}
    B_{n, k}(\coord) = \frac{1}{\deltacoord}\begin{cases}
        (\coord - \coord_{n, k-1}), \quad &\text{if } \coord\in [\coord_{n, k-1}, \coord_{n, k}],\\
        (\coord_{n, k+1} - \coord), \quad &\text{if } \coord\in [\coord_{n, k}, \coord_{n, k+1}],\\
        0, &\text{else.}
    \end{cases}
\end{equation*}
The discretized operators $\Anbase$ and $\Bnshort$ are given by \eqref{eq:AnBn_DirichletHeat}. Further, to verify all prerequisites of \Cref{thm:bound_cont_approx}, we use again the intermediate discretization \sysC from \Cref{fig:bcsystems} characterized by \eqref{eq:SysC_DirichletHeat_Ainit}-\eqref{eq:SysC_DirichletHeat_descMatrix} on $\tildeXn = \R^{n+1}$ with $\tildeEn$ and $\Rn$ given by 
\begin{equation*}
    \begin{aligned}
        (\tildeEn \tilde{g}) (y) &\coloneqq \sum_{k=0}^{n} \tilde{g}_{k} B_{n, k}(\coord),\qquad & 
        (\Rn \tilde{g} )_k & \coloneqq \tilde{g}_k , \hspace{1.5em} 1 \le k \le n-1,
    \end{aligned}
\end{equation*}
where $B_{n,k}$ was extended to the boundary by
\begin{equation*}
    B_{n, 0}= \begin{cases}
        (\coord_{n, 1}- \coord), & \text{ if } \coord \in [\coord_{n, 0}, \coord_{n, 1}]\\ 
        0 & \text{else, }
    \end{cases} \quad \text{ and }
    B_{n,n} = \begin{cases}
        (\coord_{n, n}- \coord), & \text{ if } \coord \in [\coord_{n, n-1}, \coord_{n, n}]\\ 
        0 & \text{else. }
    \end{cases}
\end{equation*}
With these preparations, we now verify the remaining assumptions of \Cref{thm:bound_cont_approx}.
\begin{enumerate}[label=(A\arabic*)]
    \item The sequence of Banach spaces $(\Xn)_{n\in\N}$ fulfills \Cref{ass:approximating-banach-spaces} as shown in \cite[Ex. 4.1, Case 2]{ItoK98}.
    \item The operators $-\Abase$ and $-\Anbase$ are self-adjoint and positive, and thus sectorial of angle $\sectorbound =0$. Hence, for $\sectorboundmoved \in (\sectorbound, \pi)$, we obtain
    \begin{align*}
        \sup_{\lambda \in \C \setminus \overline{\othersector_{\sectorboundmoved}}} \|\lambda \res{\lambda}{-\Abase}\|_X &\le \sup_{\lambda \in \C \setminus \overline{\othersector_{\sectorboundmoved}}, \sigma \in \othersector_{\sectorbound}} \|\frac{\lambda}{\lambda-\sigma}\|_X = \sup_{\lambda \in \C \setminus \overline{\othersector_{\sectorboundmoved}}, \sigma \in \othersector_{\sectorbound}} \|\frac{1}{1-\tfrac{\sigma}{\lambda}}\|_X \\
        &= \sup_{v \in \C \setminus \overline{\othersector_{\sectorboundmoved-\sectorbound }}} \|\frac{1}{1-v}\|_X \le \frac{1}{\sin(\sectorboundmoved-\sectorbound)},
    \end{align*}
    which also applies for $-\Anbase$ for all $n\in\N$, since the estimation relies only on the spectrum lying in $\othersector_\sectorbound$. In the last line, we used that the distance of $v\in \C \setminus \overline{\othersector_{\sectorboundmoved-\sectorbound }}$ from 1 is uniformly bounded from below by $\sin(\sectorboundmoved-\sectorbound)$. Hence, the operators $-\Abase$ and $-\Anbase$ for $n \in \N$ are uniformly sectorial of angle $\sectorbound =0$.
    \item The sequence of semigroups $\{\Sn(t)\}_{t\ge 0}$ generated by $\Anbase$ converges against $\{\S(t)\}_{t\ge 0}$ generated by $\Abase$ according to \cite[Lem. 4.1]{HilU25b}.
    \item Norm-convergence of the resolvents follows analogously to \cite[Thm. 3.1]{FujM76}. 
\end{enumerate}
\begin{enumerate}[label=(B\arabic*)]
    \item Follows immediately from (A2).
    \item By construction, the operator is bounded into $D(\Ainit) = H^2(\Omega)$ and we can easily verify that $\bopnrinv \in \linop{\R^2}{D(\Abase^{1/2})} = \linop{\R^2}{H^1(\Omega)}$.
    \item To show the consistency of the approximation we have to show boundedness of two terms. In the following we use the generic, $n$-independent constant $C>0$ for all estimations. 
    \begin{itemize}
        \item $\| (\Abase - \En \Anbase \Pn ) \Abase^{-1} \state \|_{X} \le C \|\state\|_X$, which follows from elliptic regularity. Let $\Abase f = \state$, then 
        \begin{align*}
            \hspace{1.7cm} \|(\Abase - \En \Anbase \Pn )\Abase^{-1} \state \|_{X} &=\| (\Abase - \En \Anbase \Pn ) f \|_{X} \le C n^{-2} \| \Abase f \|_X \\& \leq C \| \state \|_X.
        \end{align*}
        \item We show $\| \Abase^{-1} (\Amin - \En \Anbase \Pn) \state \|_{X} \le C \|\state\|_X$ and yield that by duality $\Abase \in \linop{L^2(0,1)}{H^{-2}(0,1)}$ is a bijective bounded operator such that 
        	\begin{align*}
        		\Abase^{-1} \in \linop{H^{-2}(0,1)}{L^2(0,1)}.
        	\end{align*}
        	Similarily, we can show for the discretized operators
        	\begin{align*}
        		\Anbase \Pn \in \linop{L^2(0,1)}{H_n^{-2}(0,1)},
        	\end{align*}
        	where $H_n^{-2}(0,1)$ is $\R^{n-1}$ with a suitably weighted $l_2$-norm such that the growth of $\Anbase$ with n is canceled out. Since $\En$ maps boundedly from $H_n^{-2}(0,1)$ to $H^{-2}(0,1)$, we find that $\En \Anbase \Pn$ is bounded from $L^2(0,1)$ to $H^{-2}(0,1)$ uniformly in $n$ and hence the difference is bounded as well.
    \end{itemize}
    \item Since the negative Laplace operator with homogeneous Dirichlet boundary conditions is an injective, positive, and self-adjoint operator on a Hilbert space, $\Abase \in \BIP(X)$ follows directly from \cite[Ex.~7.3.2]{MarS01}.
    \item Using
    \begin{equation*}
        \bopnrinv \coloneqq \frac{1}{n}\begin{bmatrix}0 & 1 & 2 & \dots & n \\
            n& n-1 & \dots & 1 & 0
        \end{bmatrix}\qquad\text{for all $n\in\N$},
    \end{equation*}
    we find, using the piecewise linear interpolation $\tildeEn$, that $\tildeEn \bopnrinv - \boprinv = 0$ for all $n\in \N$ and hence the condition is fulfilled.
    \item By construction we have $\Aninit \bopnrinv = 0$.
    \item Follows immediately from $\Ainit \boprinv = 0$, which is due to the fact that the second derivative of a linear function is identically zero.
\end{enumerate}

As a consequence of checking all these conditions, we conclude that  \ref{thm:ass:conv-bopnrinv} is satisfied, i.e., that $\En \Bnshort$ converges in the strong operator topology to $\Bshort$ (with $\alpha = 1/2$) and further \Cref{thm:dist_contr} to approximate \ISS gains from approximations. In \Cref{fig:heat_dirichlet}, one can numerically observe the convergence of the required parameters, and we use $\hat{\omega} = 9.8647$, $\hat{D} = 0.9991$, and $\| (-\Anbase)^\alpha \Bnshort \|_{\linop{U}{\Xn}} \leq 1.4136$ to obtain 
\begin{equation*}
    K_1 \approx 3.1408, \quad \quad K_2 \approx 0.5626, \quad\text{ and }\quad \kappa \approx 0.6359.
\end{equation*}
Consequently, \Cref{thm:dist_contr} implies that valid choices of the \ISS gains for the heat equation with control on both boundaries are given by 
\begin{equation*}
    \beta(s, t) = \mexp^{-\pi^2 t}\, s , \text{ and } \gamma(s) = 0.8989 \, s.
\end{equation*}
In particular the function $\gamma$ closely resembles the result from \cite{JacNPS18}, where rigorous analysis for one-sided control was used to determine $\gamma(s) = \sqrt{\tfrac{2}{3}} s \approx 0.8165$.

\begin{figure}[ht]
    \centering
    \begin{subfigure}[t]{0.3\linewidth}
        \begin{tikzpicture}
            \begin{axis}[
                xlabel={n},
                ylabel={$-\omega_n$},
                xmin=0, xmax=4000,
                xtick distance = 2000,
                ymin=-9.9, ymax=-9.7, 
                legend pos=north east,
                ymajorgrids=true,
                grid style=dashed,
                width=\linewidth,
                height=\linewidth,
                tick label style={font=\footnotesize},
                label style={font=\footnotesize},
                legend style={font=\footnotesize}
            ]
            
            \addplot[line width=1pt,solid,color=kit-blue] %
            table[x=n,y=omegan,col sep=comma]{heat_dirichlet_1d.csv};
            \addlegendentry{$-\omega_n$}
            \addplot[
                color=kit-green,
                ]
                coordinates {
                (0, -9.8696044) (4000, -9.8696044)
                };
                \addlegendentry{$-\hat{\omega}$}
            \end{axis}
        \end{tikzpicture}
        \caption{ The growth bound of the discretized semigroups $\|\Sn(t)\|_{\Xn} \le \mexp^{-\omega_n t}$. The reference growth bound $-\hat{\omega}$ \cite[VI.8.9]{EngN00}.}
        \label{fig:heat_equation_growth_bound}
    \end{subfigure}%
    \hfill
        \begin{subfigure}[t]{0.3\linewidth}
        \begin{tikzpicture}
            \begin{axis}[
                xlabel={n},
                ylabel={$D_n$},
                xmin=0, xmax=4000,
                xtick distance = 2000,
                yticklabel style={/pgf/number format/.cd,fixed,precision=4},
                ymin= 0.999, ymax=0.9992,
                legend pos=north east,
                ymajorgrids=true,
                grid style=dashed,
                width=\linewidth,
                height=\linewidth,
                tick label style={font=\footnotesize},
                label style={font=\footnotesize},
                legend style={font=\footnotesize}
            ]
            \addplot[line width=1pt,solid,color=kit-blue] %
            table[x=n,y=Dn,col sep=comma]{heat_dirichlet_1d.csv};
            \addlegendentry{$D_n$}
            \end{axis}
        \end{tikzpicture}
        \caption{The resolvent bounds $D_n = \sup_{\lambda \in \R_+} (\lambda+1)\|\res{\lambda}{-\Anbase}\|_{\Xn}$.}
        \label{fig:resolvent-bound-heat-equation}
    \end{subfigure}%
    \hfill
    \begin{subfigure}[t]{0.3\linewidth}
        \begin{tikzpicture}
            \begin{axis}[
                xlabel={n},
                ylabel={$\| (-\Anbase)^\alpha \Bnshort \|_{\linop{U}{\Xn}}$},
                xtick distance = 2000,
                xmax = 4000,
                ymajorgrids=true,
                grid style=dashed,
                xmin = 0,
                width=\linewidth,
                height=\linewidth,
                tick label style={font=\footnotesize},
                label style={font=\footnotesize},
                legend style={font=\footnotesize}
            ]
            \addplot[line width=\lineWidth,solid,color=kit-blue] 
            table[x=n,y=AnalphaBnnorm,col sep=comma]{heat_dirichlet_1d.csv};
            \end{axis}
        \end{tikzpicture}
        \caption{The boundary operator norm in the fractional power space $\| (-\Anbase)^\alpha \Bnshort \|_{\linop{U}{\Xn}}$. \label{fig:bn}}
    \end{subfigure}
    \caption{Values for the growth bound, the resolvent bound, and the norm of the boundary operator in the fractional power space for various equidistant discretizations of $[0, 1]$ into $n$ intervals.}
    \label{fig:heat_dirichlet}
\end{figure}
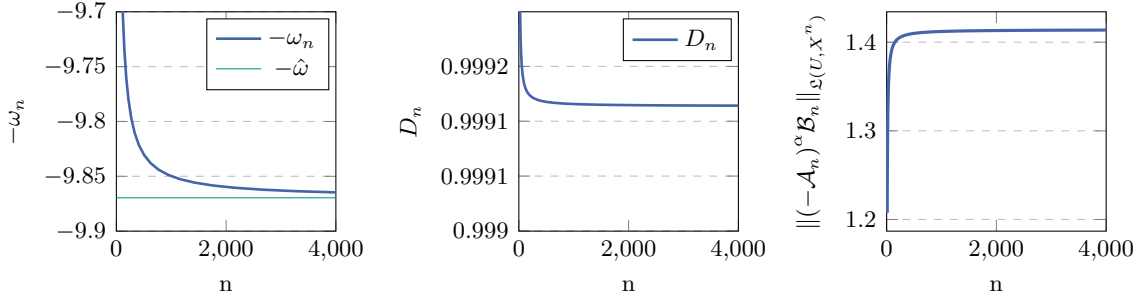

\section{Discussion}
\label{sec:discussion}

We showed for an $\Linfty$-admissible analytic control system with an unbounded control operator how to obtain explicit \ISS-gain estimates relying on the properties of a sequence of finite-dimensional approximations. The required assumptions consist of Trotter--Kato-type conditions together with additional convergence requirements for both the control operator and the generator’s resolvent. Moreover, we showed that when the unbounded control operator is the result of the Fattorini trick applied to a boundary control system, we can conclude the convergence of the control operators in the interior of the domain based on the convergence of a pre-boundary closure discretisation. 

Before continuing to focus on limitations and future research directions in the context of \ISS, we put the introduction of the pre-boundary closure discretization into a broader context. The pre-boundary closure discretization offers a first step towards a systematic approach of switching between the boundary trace formulation and the abstract Cauchy problem from a numerical analysis point of view. Both views have distinct advantages: the boundary trace formulation is well-suited for energy estimates, regularity arguments, and the explicit treatment of boundary data, while the abstract Cauchy problem setting is beneficial for operator-theoretic approaches, semigroup theory, and the design of robust controls. Hence, having a rigorous bridge between them might be advantageous also for other applications beyond \ISS. 

The Riccati operator theory provides an alternative way to obtain \ISS gains for linear infinite-dimensional systems \cite{LasT00}. This approach requires solving an operator Riccati equation, which is in general a challenging task (cf.~\cite[Ch.~9, Ex.~9.2.14]{CurZ20}). By contrast, the presented approach avoids using the Riccati equation and directly approximates the required parameters from the system semigroup and the control operator. A natural next step is to explore whether the discretisation framework for boundary control systems can be combined with the Riccati-based approximation results of \cite{LasT00}.

For future work, we envision broadening the applicability of our results to encompass more general boundary conditions, to address hyperbolic problems, and to include nonlinear or non-autonomous systems.
To elaborate the first point further, a systematic investigation of the assumptions (A*) and (B*) for Neumann or mixed boundary conditions is required. For conforming finite elements in 2D with Dirichlet boundaries, a norm-resolvent convergence result is available, for example, in \cite{FujM76}, which is missing or not easily accessible for applied mathematicians for other types of boundary conditions. 
Secondly, the approximation results can not be applied to hyperbolic problems due to the restriction to sectorial operators. However, specific hyperbolic problem groups might possess sufficient regularity properties to allow adapting the presented approach. 
Finally, extensions in the direction of semilinear systems and in the direction of non-autonomous systems, for which an extension of the Fattorini trick is readily available \cite{JacL21}, are of interest for applications.

\bibliographystyle{plain-doi} 
\bibliography{journalAbbr,literature}

\subsection*{Acknowledgments} The authors thank Prof.~Birgit~Jacob for valuable input and feedback on first ideas and Prof.~Roland~Schnaubelt for discussions on later versions of the results. Both authors acknowledge funding by the Deutsche Forschungsgemeinschaft (DFG, German Research Foundation) – Project-ID 258734477 – SFB 1173. BH additionally acknowledges funding from the DFG under Germany’s Excellence Strategy – EXC 2075 – 390740016 and from the International Max Planck Research School for Intelligent Systems (IMPRS-IS). Major parts of this manuscript were written while both authors were affiliated with the University of Stuttgart. GPT-5 was used to improve language and readability.

\subsection*{Credit author statement} BH: Conceptualization, methodology, formal analysis, writing - original draft; BU: Funding acquisition, writing - review \& editing.

\end{document}